\documentclass[a4paper]{amsart}

\usepackage{graphicx}

\usepackage{geometry}          
\usepackage{amsmath, amstext, ,amscd, amssymb, mathrsfs}
\usepackage{stmaryrd}
\usepackage{mathrsfs}
\usepackage{enumitem}

\newcommand{\R}{\mathbb{R}}
\newcommand{\C}{\mathbb{C}}
\newcommand{\T}{\mathcal{T}}

\newcommand{\N}{\mathbb{N}}

\newcommand{\restreint}[1]{\,\rule[-0,1cm]{0,1mm}{3,5mm}{}_{\,#1}}

\newcommand{\smooth}{\mathscr{C}^{\infty}}

\newcommand{\ic}{\sqrt{-1}}

\newcommand{\bw}{\overline{w}}

\newcommand{\der}[2]{\frac{\mathrm{d}#1}{\mathrm{d}#2}}
\newcommand{\derpar}[2]{\frac{\partial#1}{\partial#2}}

\newcommand{\e}{\varepsilon}

\newcommand{\E}{\mathbb{E}}
\newcommand{\EE}{\mathrm{E}}
\newcommand{\FF}{\mathscr{F}}

\newcommand{\dE}{\bar{\partial}^E}
\newcommand{\db}{\bar{\partial}}
\newcommand{\Wedge}{\Lambda^{0,\bullet}}
\newcommand{\ahf}{\Omega^{0,\bullet}}
\newcommand{\End}{\mathrm{End}}
\newcommand{\Sp}{\mathrm{Sp}}
\newcommand{\cdc}{\nabla^{\mathrm{Cl}}}
\newcommand{\LL}{\mathscr{L}}
\newcommand{\0}{\mathcal{O}}
\newcommand{\rad}{\mathcal{R}}
\newcommand{\PP}{\mathscr{P}}
\newcommand{\CC}{\mathrm{C}}
\newcommand{\num}{\mathcal{N}}
\newcommand{\RR}{\mathscr{R}}
\newcommand{\kahler}{\text{K\"{a}hler }}

\DeclareMathOperator{\tr}{Tr}

\newtheorem{prop}{Proposition}[section]
\newtheorem{thm}[prop]{Theorem}

\newtheorem{lemme}[prop]{Lemma}

\newtheorem{cor}[prop]{Corollary}

\theoremstyle{definition}

\newtheorem{defn}[prop]{Definition}

\theoremstyle{remark}

\newtheorem{rem}[prop]{Remark}

\numberwithin{equation}{section}
\begin{document}

\title{The first terms in the expansion of the Bergman kernel in higher degrees}

\title{The first terms in the expansion of the Bergman kernel in higher degrees}
\date{October 8, 2012}

\author{Martin {\sc{Puchol}}} 
\address{Universit\'{e} Paris Diderot--Paris 7, Campus des Grands Moulins, B\^{a}timent Sophie Germain, case 7012, 75205 Paris Cedex 13}
\email{martin.puchol@imj-prg.fr}

\author{Jialin {\sc{Zhu}}}
\address{Chern Institute of Mathematics, Nankai University, Tianjin China}
\email{jialinzhu@nankai.edu.cn}

\maketitle

\begin{abstract}
We establish the cancellation of the first $2j$ terms in the diagonal asymptotic expansion of the restriction to the $(0,2j)$-forms of the Bergman kernel associated to the spin${}^c$ Dirac operator on high tensor powers of a positive line bundle twisted by a (non necessarily holomorphic) complex vector bundle, over a compact K\"{a}hler manifold. Moreover, we give a local formula for the first and the second (non-zero) leading coefficients, as well as for the third assuming that the first two vanish.
\end{abstract}

\section{Introduction}
\label{intro}

The Bergman kernel of a \kahler manifold endowed with a positive line bundle $L$ is the smooth kernel of the orthogonal projection on the kernel of the Kodaira Laplacian  $\Box^L =\db^L \db^{L,*} + \db^{L,*}\db^L$. The existence of a diagonal asymptotic expansion of the Bergman kernel associated with the $p^{th}$ tensor power of $L$ when $p\to + \infty$ and the form of the leading term were proved in \cite{MR1064867}, \cite{MR1699887} and \cite{MR1616718}. Moreover, the coefficients in this expansion encode geometric information about the underlying manifold, and therefore they have been studied closely: the second and third terms were computed by Lu \cite{MR1749048}, X. Wang~\cite{MR2154820}, L. Wang \cite{MR2717131} and  by Ma-Marinescu \cite{MR2876259} in different degrees of generality (see also the recent paper \cite{Xu:2011fk}). This asymptotics plays an important role in various problems of \kahler geometry; see for instance \cite{MR1916953} or \cite{Fine:2010fk}. We refer the reader to the book \cite{ma-marinescu} for a comprehensive study of the Bergman kernel and its applications. See also the survey \cite{MR2827819}.

 In fact, Dai-Liu-Ma established the asymptotic development of the Bergman kernel in the symplectic case in \cite{MR2215454}, using the heat kernel (cf. also Ma-Marinescu \cite{MR2246888}). Recently, this asymptotics in the symplectic case found an application in the study of variation of Hodge structures of vector bundles by Charbonneau and Stern in \cite{Charbonneau:2011uq}. In their setting, the Bergman kernel is the kernel of a Kodaira-like Laplacian on a twisted bundle $L\otimes E$, where $E$ is a (not necessarily holomorphic) complex vector bundle. Because of that, the Bergman kernel is no longer supported in degree~0 (unlike it did in the \kahler case), and the asymptotic development of its restriction to the $(0,2j)$-forms is related to the degree of \textquoteleft non-holomorphicity\textquoteright   of $E$. 

In this paper, we will show that the leading term in the asymptotics of the restriction to the $(0,2j)$-forms of the Bergman kernel is of order $p^{\dim X-2j}$ and we will compute it. That will lead to a local version of \cite[(1.3)]{Charbonneau:2011uq}, which is  the main technical results of their paper; see Remark \ref{onretrouveCS}. After that, we will also compute the second term in this asymptotics, as well as the third term in the case where the first two vanish.

We now give more detail about our results. Let $(X,\omega,J)$ be a compact K\"{a}hler manifold of complex dimension $n$. Let $(L,h^L)$ be a holomorphic Hermitian line bundle on $X$, and $(E,h^E)$ a Hermitian complex vector bundle. We endow $(L,h^L)$ with its Chern (i.e., holomorphic and Hermitian) connection $\nabla^L$, and $(E,h^E)$ with a Hermitian connection $\nabla^E$, whose curvatures are respectively $R^L=( \nabla^L)^2$ and $R^E=( \nabla^E )^2$.

Except in the beginning of Section \ref{carreDp}, we will always assume that $(L,h^L,\nabla^L)$ satisfies the \emph{pre-quantization condition}:

\begin{equation}
\label{prequantization}
\omega = \frac{\ic}{2\pi} R^L.
\end{equation}

Let $g^{TX}(\cdot , \cdot) = \omega( \cdot , J\cdot )$ be the Riemannian metric on $TX$ induced by $\omega$ and $J$. It induces a metric $h^{\Wedge}$ on $\Wedge(T^*X):=\Lambda^\bullet(T^{*(0,1)}X)$, see Section \ref{carreDp}.

Let $L^p=L^{\otimes p}$ be the $p^{th}$ tensor power of $L$. Let $\ahf (X, L^p \otimes E) = \smooth(X, \Wedge (T^*X) \otimes L^p \otimes E)$ and $\db^{L^p\otimes E} : \ahf(X,L^p \otimes E)\to \ahf{}^{+1}(X,L^p \otimes E)$ be the Dolbeault operator induced by the $(0,1)$-part of $\nabla^E$ (cf. \eqref{defndbarE}). Let $\db^{L^p\otimes E,*}$ be its dual with respect to the $L^2$-product. We set (see \eqref{defndeDE}):
\begin{equation}
D_p = \sqrt{2} \Big(\db^{L^p\otimes E} +\db^{L^p\otimes E,*}\Big),
\end{equation}
which exchanges odd and even forms.

\begin{defn}
Let 
\begin{equation}
P_p \colon \ahf (X,L^p\otimes E) \to \ker(D_p)
\end{equation}
be the orthogonal projection onto the kernel $\ker(D_p)$ of $D_p$. The operator $P_p$ is called the \emph{Bergman projection}. It has a smooth kernel with respect to $dv_X(y)$, denoted by $P_p(x,y)$, which is called the \emph{Bergman kernel}.
\end{defn}

\begin{rem}
If $E$ is holomorphic, then by Hodge theory and the Kodaira vanishing theorem (see respectively \cite[Theorem 1.4.1]{ma-marinescu} and \cite[Theorem 1.5.6]{ma-marinescu}), we know that, for $p$ large enough, $P_p$ is the orthogonal projection $\smooth(X,L^p\otimes E) \to H^0(X, L^p\otimes E)$. Here, by \cite[Theorem 1.1]{Ma:2002fk}, we just know that $\ker( D_p\restreint{\Omega^{0, \text{odd}}(X,L^p\otimes E)})=0$ for $p$ large, so that $P_p \colon  \Omega^{0, \text{even}} (X,L^p\otimes E) \to \ker(D_p)$. In particular, $P_p(x,x) \in \smooth \left(X, \End(\Lambda^{0,\text{even}}(T^*X)\otimes E)\right)$.
\end{rem}

By Theorem \ref{DpestloperateurdeDiracspinc}, $D_p$ is a Dirac operator, which enables us to apply the following result:
\begin{thm}[{Dai-Liu-Ma, \cite[ Thm. 1.1]{MR2215454}}]
\label{DLM2006}
There exist $\boldsymbol{b}_r \! \in\! \smooth \! \left( X,\End(\Lambda^{0,\mathrm{even}}(T^*X)\otimes E)\right)$ such that for any $k\in \N$ and for $p\to +\infty$:
\begin{equation}
\label{DLM2006eq}
p^{-n} P_p(x,x) = \sum_{r=0}^{k} \boldsymbol{b}_r(x)p^{-r} + O(p^{-k-1}),
\end{equation}
that is for every $k,l \in \N$, there exists a constant $C_{k,l}>0$ such that for any $p\in \N$, 
\begin{equation}
\left | p^{-n} P_p(x,x) - \sum_{r=0}^k \boldsymbol{b}_r(x)p^{-r} \right |_{\mathscr{C}^l(X)} \leq C_{k,l} p^{-k-1}.
\end{equation}
Here $|\cdot|_{\mathscr{C}^l(X)}$ is the $\mathscr{C}^l$-norm for the variable $x \in X$.
\end{thm}

To simplify the formulas, we will denote by 
\begin{equation}
\label{defRscript}
\RR= (R^E)^{0,2}\in \Omega^{0,2}(X, \End(E))
\end{equation}
the $(0,2)$-part of $R^E$ (which is zero if $E$ is holomorphic). For $j\in \llbracket 1,n \rrbracket$, let 
\begin{equation}
I_j \colon \Wedge (T^*X) \otimes E \to \Lambda ^{0,j}(T^*X) \otimes E
\end{equation}
be the natural orthogonal projection.  The first main result in this paper is
\begin{thm}
\label{thm1}
For any $k\in \N$, $k\geq 2j$, we have when $p\to +\infty$:
\begin{equation}
\label{thm1eq}
p^{-n} I_{2j} P_p(x,x) I_{2j} = \sum_{r=2j}^{k} I_{2j} \boldsymbol{b}_r(x) I_{2j} \: p^{-r} + O(p^{-k-1}),
\end{equation}
and moreover,
\begin{equation}
\label{2epartiethm1}
I_{2j} \boldsymbol{b}_{2j} (x) I_{2j} = \frac{1}{(4\pi)^{2j}}\frac{1}{2^{2j}(j!)^2} I_{2j}\left( \RR_x ^ j \right)\left( \RR_x ^ j \right)^* I_{2j},
\end{equation}
where $\left( \RR_x ^ j \right)^*$ is the dual of $\RR_x^j$ acting on $(\Wedge(T^*X)\otimes E)_x$.
\end{thm}

Theorem \ref{thm1} leads immediately to
\begin{cor}
Uniformly in $x\in X$, when $p\to +\infty$, we have
\label{dl1termecor}
\begin{equation}
\label{dl1terme}
\tr \big (\left(I_{2j} P_p I_{2j}\right)(x,x)\big) =  \frac{1}{(4\pi)^{2j}}\frac{1}{2^{2j}(j!)^2} \left \| \RR_x ^j  \right \|^2 p^{n-2j}+ O(p^{n-2j-1}).
\end{equation}
\end{cor}

\begin{rem}
\label{onretrouveCS}
By integrating \eqref{dl1terme} over $X$, we get
\begin{equation}
\label{dl1termeint}
\tr \left(I_{2j} P_p I_{2j}\right)=  \frac{1}{(4\pi)^{2j}}\frac{1}{2^{2j}(j!)^2} \left \| \RR ^j  \right \|^2_{L^2} p^{n-2j}+ O(p^{n-2j-1}),
\end{equation}
which is the main technical result of Charbonneau and Stern \cite[(1.3)]{Charbonneau:2011uq}; thus Corollary \ref{dl1termecor} can be viewed as a local version of \cite[(1.3)]{Charbonneau:2011uq}.
The constant in \eqref{dl1termeint} differs from the one in \cite{Charbonneau:2011uq} because our conventions are not the same as theirs (e.g., they choose $\omega = \ic R^L$, etc...).
\end{rem}

Let  $R^E_\Lambda:=  -\ic \sum_i R^E(w_i,\bw_i)$ for $(\bw_1,\dots,\bw_n)$ an orthonormal frame of $T^{(0,1)}X$. Let $R^{TX}$ be the curvature of the Levi-Civita connection $\nabla^{TX}$ of $(X,g^{TX})$, and for $(e_1,\dots,e_{2n})$ an orthonormal frame of $TX$, let $r^X=-\sum_{i,j}\langle R^{TX}(e_i,e_j)e_i,e_j \rangle$ be the scalar curvature of $X$.

For $j,k\in \N$ and $j \geq k$, we also define $\CC_j(k)$ by
\begin{equation}
\label{defCjk}
\CC_j(k):=\frac{1}{(4\pi)^j}\frac{1}{2^{k}k!}\frac{1}{\prod_{s=k+1}^{j}(2s+1)},
\end{equation}
with the convention that $\prod_{s\in \emptyset} =1$.

Let $\nabla^{\Wedge}$ be the connection on $\Wedge(T^*X)$ induced by $\nabla^{TX}$. Let $\nabla^{\Wedge \otimes E}$ be the connection on $\Wedge(T^*X)\otimes E$ induced by $\nabla^E$ and  $\nabla^{\Wedge}$, and let $\Delta^{\Wedge \otimes E}$ be the associated Laplacian. For  precise definitions, see Section \ref{carreDp}.

For every operators $A$ acting on a Hermitian space, we will define  $\mathrm{Pos}[A]$ (resp. $\mathrm{Sym}[A]$) the positive (non necessarily definite) operator (resp. the symmetric operator) associated to $A$:
  \begin{equation}
  \mathrm{Pos}[A]=AA^* \quad \text{and} \qquad \mathrm{Sym}[A]=A+A^*.
  \end{equation}
Finally, to simplify the notation, we will define $\T_0(j)$, $\T_1(j)$, $\T_2(j)$ and $\T_3(j)$ as follows:
\begin{itemize}[label=$\bullet$]
\item 
$\T_0(0)=0$, and for $j\geq 1$,
\begin{equation}
\T_0(j)=  \frac{1}{\sqrt{2\pi}} \sum_{i=0}^{n} \sum_{k=0}^{j-1} I_{2j} \Big(\CC_j(k)-\CC_j(j)\Big)\RR_{x}^{ j-k-1}(\nabla^{\Wedge \otimes E}_{\bw_i}\RR_\cdot)(x)  \RR_{x}^{ k}I_0.
\end{equation}
\item $\T_1(0)=\T_1(1)=0$, and for $j\geq 2$,
 \begin{align}
&\T_1(j)=  \frac{I_{2j}}{2\pi}    \!  \sum_{q=0}^{j-2}\sum_{m=0}^{q} \! \left \{ \Big(\CC_j(j)-\CC_j(q+1) \Big)  \RR_{x}^{j-(q+2)}(\nabla^{\Wedge \otimes E}_{\bw_i}\RR_\cdot)(x)\RR_{x}^{q-m}(\nabla^{\Wedge \otimes E}_{w_i}\RR_\cdot)(x)\RR_{x}^{m} \right.  \notag \\
& +\CC_j\left(m \right) \left[\prod_{s=q+2}^j\Big(1+\frac{1}{2s}\Big)-1 \right]  \left. \RR_{x}^{j-(q+2)}(\nabla^{\Wedge \otimes E}_{w_i}\RR_\cdot)(x)\RR_{x}^{q-m}(\nabla^{\Wedge \otimes E}_{\bw_i}\RR_\cdot)(x)\RR_{x}^{m} \right\}I_0,
\end{align}
\item $\T_2(0)=0$, and for $j\geq 1$,
\begin{equation}
\T_2(j)= \frac{1}{4\pi}  I_{2j}  \sum_{k=0}^{j-1}\left \{ \Big (\CC_j(j)- \CC_j(k)\Big) \RR_{x}^{j-(k+1)}(\Delta^{\Wedge \otimes E}\RR_\cdot)(x) \RR_{x}^k \right\}I_0,\end{equation}
\item for $j\geq 0$,
 \begin{equation}
\T_3 (j)=I_{2j}\sum_{k=0}^j \RR_{x}^{ j-k} \left[\frac{1}{6}\left(\CC_{j+1}(j+1)-\frac{\CC_j(k)}{2\pi(2k+1)}\right)r^X_{x} -\frac{\CC_j(k)}{4\pi(2k+1)}\ic R^{E}_{\Lambda,x} \right] \RR_{x}^k I_0.
 \end{equation}
\end{itemize}

The second goal of this paper is to compute the second term in the expansion \eqref{thm1eq}.
\begin{thm}
\label{thm2}
We can decompose $I_{2j} \boldsymbol{b}_{2j+1} (x) I_{2j}$ as the sum of four terms:
 \begin{equation}
I_{2j} \boldsymbol{b}_{2j+1} (x) I_{2j} = \mathrm{Pos}[\T_0(j)] +\CC_j(j) \mathrm{Sym}\left[(\T_1(j)+\T_2(j)+\T_3(j))\left(\RR_{x}^{ j}\right) ^* I_{2j}\right].
 \end{equation}
 \end{thm}
 
For instance, for $j=1$, using the fact that $(R^E_\Lambda)^*=R^E_\Lambda$, we find
 \begin{multline}
128\pi^3 I_{2} \boldsymbol{b}_{3} (x) I_{2} =  \frac{1}{9} \mathrm{Pos}\left[I_2\sum_{i=0}^{n}(\nabla^{\Wedge \otimes E}_{\bw_i}\RR_\cdot)(x)I_0\right]-\frac{1}{6}\mathrm{Sym}\left[I_2 (\Delta^{\Wedge \otimes E}\RR_\cdot)(x)\RR_x^*I_2\right]  \\
-\frac{\ic}{6}I_2\big(R^E_\Lambda\RR_x \RR_x^*+\RR_x \RR_x^*R^E_\Lambda\big)I_2-\frac{2\ic}{3}I_2\RR_xR^E_\Lambda\RR_x^*I_2-\frac{r^X_x}{4}I_2\RR_x\RR_x^*I_2.
 \end{multline}

The last goal of this paper is to compute the third term in the expansion \eqref{thm1eq}, assuming that the first two vanish.
\begin{thm}
\label{thm3}
Let $j \in \llbracket 1, n \rrbracket$. If 
 \begin{equation}
I_{2j} \boldsymbol{b}_{2j}(x) I_{2j}=I_{2j} \boldsymbol{b}_{2j+1}(x) I_{2j}=0,
\end{equation}
then $\T_3$ equals
\begin{equation}
\T'_3 (j):=-\ic I_{2j}\sum_{k=0}^j \frac{\CC_j(k)}{4\pi(2k+1)}\RR_{x}^{ j-k}  R^{E}_{\Lambda,x} \RR_{x}^k I_0,
 \end{equation}
 and
 \begin{equation}
I_{2j} \boldsymbol{b}_{2j+2} (x) I_{2j} = \mathrm{Pos}[\T_1(j)+\T_2(j)+\T'_3(j)].
 \end{equation}
\end{thm}

Theorems \ref{thm1}, \ref{thm2} and \ref{thm3} yields to the following result.
\begin{cor}
We have:
\begin{equation}
I_{2j} P_p(x,x) I_{2j} = O(p^{n-2j-3}) \iff 
\left\{
\begin{aligned}
&\RR_x^j=0, \\
& \T_0(j)=0,\\
&\T_1(j)+\T_2(j)+\T'_3(j)=0.
\end{aligned}\right. 
\end{equation}
\end{cor}

This paper is organized as follows. In Section \ref{rescallingDp2} we compute the square of $D_p$ and use a local trivialization to rescale it, and then give the Taylor expansion of the rescaled operator. In Section~\ref{firstcoeff}, we use this expansion to give a formula for the coefficients $\boldsymbol{b}_r$ appearing in \eqref{DLM2006eq}, which will lead to a proof of Theorem \ref{thm1}. In Section \ref{2ndcoeff}, we prove Theorem \ref{thm2} using again the formula for $\boldsymbol{b}_r$.  Finally, in Section \ref{3rdcoeff}, we prove Theorem \ref{thm3} using the technics and results of the preceding sections.

In this whole paper, when an index variable appears twice in a single term, it means that we are summing over all its possible values.

\subsection*{Acknowledgments} The authors want to thank Professor Xiaonan Ma for helpful discussions about this topic, and more generally for his kind and constant support.


\section{Rescaling $D_p^2$ and Taylor expansion}
\label{rescallingDp2}

In this section, we follow the method of \cite[Chapter 4]{ma-marinescu}, that enables to prove the existence of $\boldsymbol{b}_r$ in \eqref{DLM2006eq} in the case of a holomorphic vector bundle $E$, and that still applies here (as pointed out in \cite[Section 8.1.1]{ma-marinescu}). Then, in section 3 and 4, we will use this approach to understand $I_{2j} \boldsymbol{b}_r I_{2j}$ and prove Theorems \ref{thm1} and \ref{thm2}. 

In Section \ref{carreDp}, we will first prove Theorem \ref{DpestloperateurdeDiracspinc}, and then give a formula for the square of $D_p$, which will be the starting point of our approach.

In Section \ref{rescalingDp2}, we will rescale the operator $D_p^2$ to get an operator $\LL_t$, and then give  the Taylor expansion of the rescaled operator.

In Section \ref{BergmankernelofL0}, we will study more precisely the limit operator $\LL_0$.

\subsection{The square of $D_p$}
\label{carreDp}

For further details on the material of this subsection, the lector can read \cite{ma-marinescu}. First of all let us give some notations.

The Riemannian volume form of $(X, g^{TX})$ is given by $dv_X  =\omega^n/n!$ . We will denote by  $\langle \cdot , \cdot \rangle$ the $\C$-bilinear form on $TX\otimes \C$ induced by $g^{TX}$.

 For the rest of this Section \ref{carreDp}, we will fix $(w_1,\dots,w_n)$ a local orthonormal frame of $T^{(1,0)}X$ with dual frame $(w^1,\dots,w^n)$. Then $(\bw_1,\dots,\bw_n)$ is a local orthonormal frame of $T^{(0,1)}X$ whose dual frame is denoted by $(\bw^1,\dots,\bw^n)$, and the vectors
\begin{equation}
\label{defnbase}
e_{2j-1}=\frac{1}{\sqrt{2}} \left( w_j + \bw_j \right) \text{ and } e_{2j}=\frac{\ic}{\sqrt{2}} \left( w_j - \bw_j \right)
\end{equation}
form a local orthonormal frame of $TX$.

We choose the Hermitian metric  $h^{\Wedge}$ on $\Wedge (T^*X):=\Lambda^\bullet(T^{*(0,1)}X)$  such that \\ $ \{ \bw^{j_1} \wedge \dots \wedge \bw^{j_k} / 1\leq j_1< \dots<j_k\leq n\}$ is an orthonormal frame of $\Wedge (T^*X)$.

For any Hermitian bundle $(F,h^F)$ over $X$, let $\smooth(X, F)$  be the space of smooth sections of~$F$. It is endowed with the $L^2$-Hermitian metric:

\begin{equation}
\langle s_1, s_2 \rangle = \int_X  \langle s_1(x) , s_2(x) \rangle_{h^F} dv_x(x).
\end{equation}
The corresponding norm will be denoted by $\| \cdot \|_{L^2}$, and the completion of $\smooth(X, F)$ with respect to this norm by $L^2(X,F)$.

Let $\dE$ be the Dolbeault operator of $E$: it is the $(0,1)$-part of the connection $\nabla^E$
\begin{equation}
\label{defndbarE}
\dE : = \left( \nabla^E \right)^{0,1} \colon \smooth (X,E) \to \smooth ( X , T^{*(0,1)}X\otimes E).
\end{equation}
We  extend it to get an operator 
\begin{equation}
\label{defndbarE2}
\dE \colon \ahf (X,E) \to \ahf {}^{+1}(X,E)
\end{equation}
 by the Leibniz formula: for $ s \in \smooth (X,E)$ and $\alpha \in \smooth ( X, \Wedge (T^*X))$ homogeneous,
\begin{equation}
\dE (\alpha \otimes s)  = (\db\alpha) \otimes s + (-1)^{\deg \alpha} \alpha \otimes \dE s.
\end{equation}

We can now define the operator 
\begin{equation}
\label{defndeDE}
D^E = \sqrt{2} \left( \dE + \dE{}^{,*} \right) \colon \ahf(X,E) \to \ahf (X,E),
\end{equation}
where the dual is taken with respect to the $L^2$- norm associated with the Hermitian metrics $h^{\Wedge}$ and $h^E$. 

Let  $\nabla^{\Lambda(T^*X)}$ be the connection on $\Lambda(T^*X)$ induced by the Levi-Civita connection $\nabla^{TX}$ of $X$. Since $X$ is K\"{a}hler, $\nabla^{TX}$ preserves $T^{(0,1)}X$ and $T^{(1,0)}X$. Thus, it induces a connection $\nabla^{T^{*(0,1)}X}$ on $T^{*(0,1)}X$, and then a Hermitian connection $\nabla^{\Wedge}$ on $\Wedge(T^*X)$. We then have that for any $\alpha \in \smooth(X,\Wedge(T^*X))$,
\begin{equation}
\label{nablaahf}
\nabla^{\Wedge}\alpha = \nabla^{\Lambda(T^*X)}\alpha.
\end{equation}
Note the important fact that $\nabla^{\Lambda(T^*X)}$ preserves the bi-grading on $\Lambda^{\bullet,\bullet}(T^*X)$.

Let $\nabla^{\Wedge\otimes E} :=\nabla^{\Wedge}\otimes 1 + 1\otimes \nabla^E$ be the connection on $\Wedge(T^*X)\otimes E$ induced by $\nabla^{\Wedge}$ and $\nabla^E$

\begin{prop}
On $\ahf(X,E)$, we have:
\begin{equation}
\begin{matrix}
\label{formulededeetoile}
\dE =  \bw^j\wedge \nabla^{\Wedge\otimes E}_{\bw_j}, \\
\dE{}^{,*} = -i_{\bw_j} \nabla^{\Wedge\otimes E}_{w_j}.
\end{matrix}
\end{equation}
\end{prop}

\begin{proof}
We still denote by $\nabla^E$ the extension of the connection $\nabla^E$ to $\Omega^{\bullet,\bullet}(X,E)$ by the usual formula $\nabla^E(\alpha \otimes s) = d\alpha \otimes s + (-1)^{\deg \alpha} \alpha \wedge \nabla^E s$ for $ s \in \smooth (X,E)$ and $\alpha \in \smooth ( X, \Lambda (T^*X))$ homogeneous. We know that $d = \e \circ \nabla^{\Lambda(T^*X)}$ where $\e$ is the exterior multiplication (see \cite[(1.2.44)]{ma-marinescu}), so we get that $\nabla^E=\e\circ\nabla^{\Lambda(T^*X) \otimes E}$. Using \eqref{nablaahf}, it follows that
$$ \dE = (\nabla^E)^{0,1} = \bw^j\wedge \nabla^{\Wedge\otimes E}_{\bw_j},$$
which is the first part of \eqref{formulededeetoile}.

The second part of our proposition follows classically from the first by exactly the same computation as in \cite[Lemma 1.4.4]{ma-marinescu}.
 \end{proof}

\begin{defn}
Let $v = v^{1,0} + v^{0,1} \in TX= T^{(1,0)}X\oplus T^{(0,1)}X$, and $\bar{v}^{(0,1),*}\in T^{*(0,1)}X$ the dual of $v^{1,0}$ for $\langle \cdot , \cdot \rangle$. We define the \emph{Clifford action of  $TX$ on $\Wedge(T^*X)$} by
\begin{equation}
c(v) = \sqrt{2} \left( \bar{v}^{(0,1),*} \wedge\, - i_{v^{0,1}} \right).
\end{equation}
\end{defn}
We verify easily that for $u,v \in TX$,
\begin{equation}
c(u)c(v) + c(v)c(u) = -2 \langle u,v\rangle,
\end{equation}
and that for any skew-adjoint endomorphism $A$ of $TX$,
\begin{align}
\label{formulecA}
\frac{1}{4} \langle Ae_i,e_j \rangle c(e_i)&c(e_j) = -\frac{1}{2}\langle Aw_j,\bw_j\rangle + \langle Aw_\ell,\bw_m \rangle \bw^m \wedge i_{\bw_\ell} \notag \\
 &+\frac{1}{2}\langle Aw_\ell,w_m\rangle i_{\bw_\ell} i_{\bw_m}+\frac{1}{2}\langle A\bw_\ell,\bw_m \rangle \bw^\ell\wedge\bw^m\wedge.
\end{align}

Let $\nabla^{\det}$ be the Chern connection of $\det(T^{(1,0)}X) := \Lambda^n (T^{(1,0)}X)$, and $\nabla^{\mathrm{Cl}}$ the Clifford connection on $\Wedge(T^*X)$ induced by $\nabla^{TX}$ and $\nabla^{\det}$ (see \cite[(1.3.5)]{ma-marinescu}). We also denote by $\nabla^{\mathrm{Cl}}$ the connection on $\Wedge(T^*X)\otimes E$ induced by $\nabla^{\mathrm{Cl}}$ and $\nabla^E$. By \cite[(1.3.5)]{ma-marinescu}, \eqref{formulecA} and the fact that $\nabla^{\det}$ is holomorphic, we get 
\begin{equation}
\label{cdc=nabla}
\nabla^{\mathrm{Cl}}= \nabla^{\Wedge}.
\end{equation}

Let $D^{c,E}$ be the associated spin${}^c$ Dirac operator:
\begin{equation}
\label{defnDc}
D^{c,E} = \sum_{j=1}^{2n} c(e_j) \cdc _{e_j} \colon \ahf(X,E) \to \ahf(X,E).
\end{equation}

By \eqref{formulededeetoile} and \eqref{cdc=nabla}, we have
\begin{thm}
\label{DpestloperateurdeDiracspinc}
$D^E$ is equal to the spin${}^c$ Dirac operator $D^{c,E}$ acting on $\ahf(X,E)$.
\end{thm}

\begin{rem}
\label{pasdHP}
Note that all the results proved in the beginning of this subsection hold without assuming the pre-quantization condition~\eqref{prequantization}, but from now on we will use it.
\end{rem}  

Let $(F,h^F)$ be a Hermitian vector bundle on $X$ and let $\nabla^F$ be a Hermitian connection on $F$. Then the \emph{Bochner Laplacian} $\Delta^F$ acting on $\smooth(X,F)$ is defined by
\begin{equation}
\Delta^F = - \sum_{j=1}^{2n} \left( (\nabla^F_{e_j})^2 -\nabla^F_{\nabla^{TX}_{e_j}e_j} \right).
\end{equation}
On $\ahf(X)$, we define the \emph{number operator} $\num$ by
\begin{equation}
\num\restreint{\Omega^{0,j}(X)} = j,
\end{equation}
and we also denote by $\num$ the operator $\num\otimes 1$ acting on $\ahf(X,F)$.

The bundle $L^p$ is endowed with the connection $\nabla^{L^p}$ induced by $\nabla^L$ (which is also its Chern connection). Let $\nabla^{L^p \otimes E} := \nabla^{L^p}\otimes 1 + 1 \otimes \nabla^E$ be the connection on $L^p \otimes E$ induced by $\nabla^L$ and $\nabla^E$. We will denote
\begin{equation}
\label{defndeDp}
D_p = D^{L^p \otimes E}.
\end{equation}

\begin{thm}
\label{carreDpprecis}
The square of $D_p$ is given by
\begin{align}
\label{carredpeq2}
D_p^2 = \Delta^{\Wedge\otimes L^p \otimes E} - &R^E(w_j,\bw_j) - 2\pi p n + 4\pi p\num +2\left( R^E +\frac{1}{2} R^{\det} \right)(w_\ell , \bw_m) \bw^m \wedge i_{\bw_\ell} \notag \\
&+ R^E(w_\ell,w_m)i_{\bw_\ell}i_{\bw_m}+ R^E(\bw_\ell,\bw_m)\bw^\ell \wedge \bw^m.
\end{align}
\end{thm}

\begin{proof}
By Theorem \ref{DpestloperateurdeDiracspinc}, we can use \cite[Theorem 1.3.5]{ma-marinescu}:
\begin{equation}
\label{lichnerowicz}
D_p^2 = \Delta^{\mathrm{Cl}} + \frac{r^X}{4} + \frac{1}{2} \left( R^{L^p \otimes E} +\frac{1}{2} R^{\det} \right) (e_i , e_j)c(e_i) c(e_j),
\end{equation}
where $r^X$ is the scalar curvature of $X$. From \eqref{cdc=nabla}, we see that $\Delta^{\mathrm{Cl}}=\Delta^{\Wedge\otimes L^p \otimes E} $. Moreover, $r^X = 2 R^{\det}(w_j, \bw_j)$ and $R^{L^p \otimes E} =R^E + pR^L$. Using the equivalent of \eqref{formulecA} for 2-forms (substituting $A(\cdot,\cdot)$ for $ \langle A\cdot, \cdot \rangle$) and the fact that $R^L$ and $R^{\det}$ are $(1,1)$-forms, \eqref{lichnerowicz} reads

\begin{align*}
D_p^2 &= \Delta^{\Wedge\otimes L^p \otimes E} +\frac{1}{2}R^{\det}(w_j,\bw_j)  - \left( R^E(w_j,\bw_j) +  p R^L(w_j,\bw_j) +\frac{1}{2}R^{\det}(w_j,\bw_j)\right) \\
&\quad + 2 \left( R^E +pR^L+\frac{1}{2} R^{\det} \right)(w_\ell , \bw_m) \bw^m \wedge i_{\bw_\ell} +  R^E (w_\ell,w_m)i_{\bw_\ell}i_{\bw_m} \\
&\quad +   R^E (\bw_\ell,\bw_m)\bw^\ell \wedge \bw^m.
\end{align*}

Thanks to \eqref{prequantization}, we have $R^L(w_\ell , \bw_m)  = 2 \pi \delta_{\ell m}$.  Moreover, $\num=\sum_\ell \bw^\ell \wedge i_{\bw_\ell}$, thus we get Theorem \ref{carreDpprecis}.
 \end{proof}

\subsection{Rescaling $D^2_p$}
\label{rescalingDp2}

In this subsection, we  rescale $D_p^2$, but to do this we must define it on a vector space. Therefore, we will use normal coordinates to transfer the problem on the tangent space to $X$ at a fixed point.  Then we  give a Taylor expansion of the rescaled operator, but the problem is that each operator acts on a different space, namely $$\EE_p :=\Wedge(T^*X)\otimes L^p \otimes E,$$ so we must first handle this issue.

Fix $x_0 \in X$. For the rest of this paper, we fix $\{ w_j\}$  an orthonormal basis of $T_{x_0}^{(1,0)}X$, with dual basis  $\{ w^j\}$, and we construct an orthonormal basis $\{ e_i\}$ of $T_{x_0}X$ from $\{ w_j\}$ as in \eqref{defnbase}.

 For $\e>0$, we denote by $B^X(x_0,\e)$ and $B^{T_{x_0}X}(0,\e)$ the open balls in $X$ and $T_{x_0}X$ with center $x_0$ and $0$ and radius $\e$ respectively. If $\exp^X_{x_0}$ is the Riemannian exponential of $X$, then for $\e$ small enough, $ Z\in B^{T_{x_0}X}(0,\e) \mapsto \exp^X_{x_0}(Z) \in B^X(x_0,\e)$ is a diffeomorphism, which gives local coordinates by identifying $T_{x_0}X$ with $\R^{2n}$ via the orthonormal basis $\{e_i \}$:

\begin{equation}
(Z_1,\dots,Z_{2n}) \in \R^{2n} \mapsto \sum_i Z_i e_i \in T_{x_0}X.
\end{equation} 
From now on, we will always identify $ B^{T_{x_0}X}(0,\e)$ and $ B^X(x_0,\e)$. Note that in this identification, the radial vector field $\rad =  \sum_i Z_i e_i$ becomes $\rad = Z$, so $Z$ can be viewed as a point or as a tangent vector.

For $Z\in B^{T_{x_0}X}(0,\e)$, we identify $(L_Z,h^L_Z)$, $(E_Z, h^E_Z)$ and $(\Wedge_Z(T^*X), h^{\Wedge}_Z)$  with $(L_{x_0},h^L_{x_0})$, $(E_{x_0}, h^E_{x_0})$ and $(\Wedge (T^*_{x_0}X), h_{x_0}^{\Wedge})$ by parallel transport with respect to the connection $\nabla^L$, $\nabla^E$ and $\nabla^{\Wedge}$ along the geodesic ray $t \in [0,1]\mapsto tZ$. We denote by $\Gamma^L$, $\Gamma^E$ and $\Gamma^{\Wedge}$ the corresponding connection forms of $\nabla^L$, $\nabla^E$ and $\nabla^{\Wedge}$.

\begin{rem}
\label{deg}
Note that since $\nabla^{\Wedge}$ preserves the degree, the identification between $\Wedge (T^*X)$ and $\Wedge (T^*_{x_0}X)$ is compatible with the degree. Thus, $\Gamma^{\Wedge}_Z \in \bigoplus_j \End(\Lambda^{0,j}(T^*X))$.
\end{rem}

Let $S_L$ be a unit vector of $L_{x_0}$. It gives an isometry $L^p_{x_0} \simeq \C$, which induces an isometry 
\begin{equation}
\EE_{p,x_0} \simeq (\Wedge (T^*X)\otimes E)_{x_0} =: \E_{x_0}.
\end{equation}
Thus, in our trivialization, $D^2_p$ acts on $\E_{x_0}$, but this action may \emph{a priori} depend on the choice of $S_L$. In fact, since the operator $D^2_p$ takes values in $\End(\EE_{p,x_0})$ which is canonically isomorphic to $\End(\E)_{x_0}$ (by the natural identification $\End(L^p) \simeq \C$), all our formulas do not depend on this choice.

Let $dv_{TX}$ be the Riemannian volume form of $(T_{x_0}X, g^{T_{x_0}X})$, and $\kappa(Z)$ be the smooth positive function defined for $|Z| \leq \e$ by
\begin{equation}
dv_{X} (Z) = \kappa(Z) dv_{TX}(Z),
\end{equation}
with $\kappa(0)=1$.

\begin{defn}
We denote by $\nabla_U$ the ordinary differentiation operator in the direction $U$ on $T_{x_0}X$. For $s\in \smooth ( \R^{2n}, \E_{x_0})$, and for $t= \frac{1}{\sqrt{p}}$, set
\begin{equation}
\label{defnrescaled}
\begin{aligned}
&(S_ts)(Z)  = s(Z/t), \\
&\nabla_t = t S_t^{-1} \kappa^{1/2} \nabla^{\mathrm{Cl_0}} \kappa^{-1/2} S_t, \\
& \nabla_0= \nabla + \frac{1}{2} R^L_{x_0}( Z, \cdot), \\
&\LL_t =  t^2 S_t^{-1} \kappa^{1/2} D_p^2 \kappa^{-1/2} S_t, \\
&\LL_0 = - \sum_i (\nabla_{0,e_i})^2 + 4\pi \num - 2\pi n.
\end{aligned}
\end{equation}
\end{defn}

Let $||\cdot||_{L^2}$ be the $L^2$-norm induced by $h^{\EE_{x_0}}$ and $dv_{TX}$. We can now state the key result in our approach to Theorems \ref{thm1} and \ref{thm2}:

\begin{thm}
\label{dldeLt}
There exist second-order formally self-adjoint (with respect to $||\cdot||_{L^2}$) differential operators $\0_r$ with polynomial coefficients such that for all $m\in \N$,
\begin{equation}
\label{dldeLteq}
\LL_t = \LL_0 +Ê\sum_{r=1}^m t^r \0_r + O(t^{m+1}).
\end{equation}
Furthermore, each $\0_r$ can be decomposed as
\begin{equation}
\label{decompodeOr}
\0_r = \0_r^0 + \0_r^{+2} + \0_r^{-2},
\end{equation}
where $\0_r^{k}$ changes the degree of the form it acts on by $k$.
\end{thm}

\begin{proof}
The first part of the theorem (i.e., equation \eqref{dldeLteq}) is contained in \cite[Theorem 1.4]{MR2382740}. We will briefly recall how they obtained this result. 

 Let $\Phi_{E}$ be the smooth self-adjoint section of $\End(\E_{x_0})$ on $B^{T_{x_0}X}(0,\e)$: 
\begin{equation}
\label{defnPhi}
\begin{aligned}
\Phi_{E} =- R^E(w_{j},\bw_{j}) +&2\left( R^E +\frac{1}{2} R^{\det} \right)(w_{\ell} , \bw_{m}) \bw^m \wedge i_{\bw_{\ell}} \\
&+ R^E(w_{\ell},w_{m})i_{\bw_{\ell}}i_{\bw_{m}}+ R^E(\bw_{\ell},\bw_{m})\bw^\ell \wedge \bw^m.
\end{aligned}
\end{equation}
We can see that we can decompose $\Phi_{E}=\Phi_{E}^0+\Phi_{E}^{+2} + \Phi_{E}^{-2}$, where
\begin{equation}
\label{decompodePhi}
\begin{aligned}
&\Phi_{E}^0 =  R^E(w_{j},\bw_{j}) +2\left( R^E +\frac{1}{2} R^{\det} \right)(w_{\ell} , \bw_{m}) \bw^m \wedge i_{\bw_{\ell}} \text{ preserves the degree,} \\
&\Phi_{E}^{+2} =R^E(\bw_{\ell},\bw_{m})\bw^\ell \wedge \bw^m \text{ rises the degree by 2,}\\
&\Phi_{E}^{-2} = R^E(w_{\ell},w_{m})i_{\bw_{\ell}}i_{\bw_{m}} \text{ lowers the degree by 2.}
\end{aligned}
\end{equation}

Using Theorem \ref{carreDpprecis}, we find that :
\begin{equation}
\label{Dp2avecPhi}
D_p^2 = \Delta^{\Wedge \otimes L^p \otimes E} + p(-2\pi n + 4\pi \num) +\Phi_{E}.
\end{equation}

Let $g_{ij}(Z) = g^{TX}(e_i,e_j)(Z)$ and $\left( g^{ij}(Z) \right)_{ij}$ be the inverse of the matrix $\left( g_{ij}(Z) \right)_{ij}$. Let $\left(\nabla^{TX}_{e_i}e_j\right)(Z)=\Gamma^k_{ij}(Z)e_k$.  As in \cite[(4.1.34)]{ma-marinescu}, by \eqref{defnrescaled} and \eqref{Dp2avecPhi}, we get:
\begin{equation}
\label{calculedeLt}
\begin{aligned}
&\nabla_{t,\cdot} = \kappa^{1/2}(tZ) \left( \nabla_\cdot + t \Gamma^{\Wedge}_{tZ} +\frac{1}{t} \Gamma^{L}_{tZ} + t\Gamma^{E}_{tZ} \right)  \kappa^{-1/2}(tZ), \\
&\LL_t = -g^{ij}(tZ) \left( \nabla_{t,e_i}\nabla_{t,e_j} - t \Gamma^k_{ij}(tZ)\nabla_{t,e_k} \right) - 2\pi n + 4\pi \num + t^2 \Phi_{E}(tZ).
\end{aligned}
\end{equation}
 Moreover, $\kappa = ( \det(g_{ij}))^{1/2}$, thus we can prove equations  \eqref{dldeLteq} as in \cite[Theorem 4.1.7]{ma-marinescu} by taking the Taylor expansion of each term appearing in \eqref{calculedeLt}. Note that in \cite{ma-marinescu}, every data has to be extended to $T_{x_0}X$ to make the analysis work, but as we admit the result, we do not have to worry about it and simply restrict ourselves to a neighborhood of $x_0$.

Now, it is clear that in the formula for $\LL_t$ in \eqref{calculedeLt}, the term 
\begin{equation}
\label{defLt0}
\LL_t^0 := -g^{ij}(tZ) \left( \nabla_{t,e_i}\nabla_{t,e_j} - t \Gamma^k_{ij}(tZ)\nabla_{t,e_k} \right) - 2\pi n + 4\pi \num +t^2\Phi_{E}^0(tZ)
\end{equation}
preserves the degree, because $\Gamma^{\Wedge}$ does (as explained in Remark \ref{deg}). Thus, using \eqref{decompodePhi} and taking Taylor expansion of $\LL_t$ in \eqref{calculedeLt}, we can write
\begin{equation}
\label{dldeLtprecis}
\begin{aligned}
& \LL_t^0=  \LL_0 +Ê\sum_{r=1}^\infty t^r \0_r^0, \\
&t^2\Phi_{E}^{\pm2}(tZ) =\sum_{r=2}^\infty t^r \0_r^{\pm2}.
\end{aligned}
\end{equation}
From \eqref{dldeLtprecis}, we get \eqref{decompodeOr}.

Finally, due to the presence of the conjugation by $\kappa^{1/2}$ in \eqref{defnrescaled}, $\LL_t$ is a formally self-adjoint operator on $\smooth(\R^{2n},\EE_{x_0})$ with respect to $||\cdot||_{L^2}$. Thus, $\LL_0$ and the $\0_r$'s also are.
 \end{proof}

Recall that $\RR= (R^E)^{0,2}\in \Omega^{0,2}(X, \End(E))$.

\begin{prop}
\label{expressionprecisedeO2}
We have
\begin{equation}
\label{annulation01}
\0_1 = 0.
\end{equation}

For $\0_2$, we have the formulas:
\begin{equation}
\label{O2+2etO2-2}
\0_2^{+2} = \RR_{x_0} \: , \quad \0_2^{-2} = \left(\RR_{x_0}\right)^*,
\end{equation}
and
\begin{multline}
\label{O20}
\0_2^0 = \frac{1}{3} \langle R^{TX}_{x_0}(Z,e_i)Z,e_j\rangle \nabla_{0,e_i}\nabla_{0,e_j}-R^E_{x_0}(w_j,\bw_j) - \frac{r^X_{x_0}}{6} \\
+ \left( \left \langle \frac{1}{3}R^{TX}_{x_0}(Z,e_k)e_k +\frac{\pi}{3}R^{TX}_{x_0} (z,\bar{z})Z,e_j\right \rangle -R^E_{x_0}(Z,e_j)\right) \nabla_{0,e_j}.
\end{multline}
\end{prop}

\begin{proof}
For $F= L,E$ or $\Wedge(T^*X)$, it is known that (see for instance \cite[Lemma 1.2.4]{ma-marinescu})
\begin{equation}
\label{dlGamma}
\sum_{|\alpha|=r} (\partial^\alpha \Gamma^F)_{x_0}(e_j)\frac{Z^\alpha}{\alpha!} =\frac{1}{r+1}\sum_{|\alpha|=r-1} (\partial^\alpha R^F)_{x_0}(Z,e_j)\frac{Z^\alpha}{\alpha!},
\end{equation}
and in particular,
\begin{equation}
\label{dlGamma2}
\Gamma^F_Z(e_j) = \frac{1}{2}R^F_{x_0}(Z,e_j) +O(|Z|^2).
\end{equation}

Furthermore, we know that
\begin{equation}
\label{dlmetric}
g_{ij}(Z) = \delta_{ij} + O(|Z|^2):
\end{equation}
it is the Gauss lemma (see \cite[(1.2.19)]{ma-marinescu}). It implies that 
\begin{equation}
\label{dlK}
\kappa(Z) = |\det(g_{ij}(Z) )|^{1/2} = 1+O(|Z|^2).
\end{equation}
Moreover, the second line of \cite[(4.1.103)]{ma-marinescu} entails
\begin{equation}
\label{dlRL}
\frac{\ic}{2\pi}R^L_{Z}(Z, e_j)= \langle JZ,e_j \rangle +O(|Z|^3),
\end{equation}
and thus by \eqref{dlGamma} and \eqref{dlRL}
\begin{equation}
\label{dlGammaL}
\Gamma^L_{Z}= \frac{1}{2}R^L_{x_0}(Z,e_j) +O(|Z|^3).
\end{equation}

Using \eqref{calculedeLt}, \eqref{dlGamma2}, \eqref{dlK} and \eqref{dlGammaL}, we see that
\begin{equation}
\label{dlnablat}
\nabla_t = \nabla_0 +O(t^2).
\end{equation}
Finally, using again \eqref{calculedeLt}, \eqref{dlmetric} and \eqref{dlnablat}, we get $\0_1=0$.

Concerning $\0_2^{\pm2}$, from \eqref{dldeLtprecis}, we see that
\begin{equation}
\begin{aligned}
&\0_2^{+2} = \Phi_{E}^{+2}(0) = R_{x_0}^E(\bw_{\ell},\bw_{m})\bw^\ell \wedge \bw^m = (R^E_{x_0})^{0,2} =\RR_{x_0},\\
&\0_2^{-2} = \Phi_{E}^{-2}(0) =  R_{x_0}^E(w_{\ell},w_{m})i_{\bw_{\ell}}i_{\bw_{m}} =  \left((R^E_{x_0})^{0,2}\right)^*=\left( \RR_{x_0} \right)^*.
\end{aligned}
\end{equation}

Finally, by \eqref{defLt0} and \cite[(4.1.34)]{ma-marinescu}, we see that our $\LL_t^0$ corresponds to $\LL_t$ in \cite{ma-marinescu}. Thus, by \eqref{dldeLtprecis} and \cite[(4.1.31)]{ma-marinescu}, our $\0_2^0$ is equal to their $\0_2$ (this is because in their case, $E$ is holomorphic, so $R^E$ is a $(1,1)$-form and there is no term changing the degree in $(\bar{\partial}^{L^p \otimes E} + \bar{\partial}^{L^p \otimes E,*})^2$, but the terms preserving the degree are the same as ours). Hence \eqref{O20} follows from  \cite[Theorem 4.1.25]{ma-marinescu}.
 \end{proof}

\subsection{Bergman kernel of the limit operator $\LL_0$}
\label{BergmankernelofL0}

In this subsection, we study more precisely the operator $\LL_0$.

We introduce the complex coordinates $z=(z_1, \dots , z_n)$ on $\C^n \simeq \R^{2n}$. Thus, we get $Z=z+\bar{z}$, $w_j=\sqrt{2}\derpar{}{z_j}$ and $\bw_j=\sqrt{2}\derpar{}{\bar{z}_j}$. We will identify $z$ to $\sum_j z_j \derpar{}{z_j}$ and $\bar{z}$ to $\sum_j \bar{z}_j \derpar{}{\bar{z}_j}$ when we consider $z$ and $\bar{z}$ as vector fields.

Set
\begin{equation}
\begin{aligned}
&b_j = -2 \nabla_{0,\derpar{}{z_j}}, 
&  b_j^+=2 \nabla_{0,\derpar{}{\bar{z}_j}},  \,\, \quad\qquad\qquad\\
&b= (b_1,\dots,b_n), 
&\quad  \LL = - \sum_i (\nabla_{0,e_i})^2  - 2\pi n.
\end{aligned}
\end{equation}

By definition, $\nabla_0 =\nabla + \frac{1}{2} R^L_{x_0}( Z, \cdot)$ so we get
\begin{equation}
\label{defbi}
b_i = -2 \derpar{}{z_i}+\pi \bar{z}_i, \quad b_i^+ = 2 \derpar{}{\bar{z}_i}+\pi z_i,
\end{equation}
and for any polynomial $g(z,\bar{z})$ in $z$ and $\bar{z}$,
\begin{equation}
\label{commutation}
\begin{matrix}
&[b_i,b_j^+]= -4\pi \delta_{ij},
&[b_i,b_j]=[b_i^+,b_j^+]=0, \\
&[g(z,\bar{z}), b_j]=2 \derpar{}{z_j}g(z,\bar{z}), 
&\quad[g(z,\bar{z}), b_j^+]=-2 \derpar{}{\bar{z}_j}g(z,\bar{z}).
\end{matrix}
\end{equation}
Finally, a simple calculation shows:
\begin{equation}
\label{LetL0}
\LL= \sum_i b_i b_i^+ \text{ and } \LL_0 = \LL+ 4\pi \num.
\end{equation}
Recall that we denoted by $||\cdot||_{L^2}$ the $L^2$-norm  associated with $h^{\E_{x_0}}$ and $dv_{TX}$. As for this norm $b_i^+ = (b_i)^*$, we see that $\LL$ and $\LL_0$ are self-adjoint with respect to this norm.

The next theorem is proved in \cite[Theorem 4.1.20]{ma-marinescu}:
\begin{thm}
\label{elementspropresdeL}
The spectrum of the restriction of $\LL$ to $L^2(\R^{2n})$ is $\Sp (\LL \restreint{L^2(\R^{2n})}) =  4\pi \N$ and an orthogonal basis of the eigenspace for the eigenvalue $4\pi k$ is
\begin{equation}
\label{vep}
b^\alpha \left( z^\beta \exp \left( -\frac{\pi}{2} |z|^2\right) \right) \text{, with } \alpha,\beta \in \N^n\text{ and } \sum_i \alpha_i=k.
\end{equation}
\end{thm}
 
 Especially, an orthonormal basis of $\ker(\LL\restreint{L^2(\R^{2n})})$ is 
\begin{equation}
\left( \frac{\pi^{|\beta|}}{\beta!} \right)^{1/2} z^\beta \exp \left( -\frac{\pi}{2} |z|^2\right),
\end{equation}
and thus if $\PP(Z,Z')$ is the smooth kernel of $\PP$ the orthogonal projection from $(L^2(\R^{2n}), ||\cdot||_0)$ onto $\ker(\LL)$ (where $||\cdot||_0$ is the $L^2$-norm associated to $g^{TX}_{x_0}$) with respect to $dv_{TX}(Z')$, we have
\begin{equation}
\label{formuleP}
\PP(Z,Z') = \exp \left( -\frac{\pi}{2} (|z|^2+|z'|^2-2z\cdot\bar{z}')\right).
\end{equation}

Now let $P^N$ be the orthogonal projection from $(L^2(\R^{2n}, \E_{x_0}), ||\cdot||_{L^2})$ onto $N:=\ker(\LL_0)$, and $P^N(Z,Z')$ be its smooth kernel with respect to $dv_{TX}(Z')$. From \eqref{LetL0}, we have:
\begin{equation}
\label{PNetP}
P^N(Z,Z') = \PP(Z,Z') I_0.
\end{equation}


\section{The first coefficient in the asymptotic expansion}
\label{firstcoeff}

In this section we prove Theorem \ref{thm1}. We will proceed as follows.

 In Section \ref{aformulaforbr}, following \cite[Section 4.1.7]{ma-marinescu}, we will give a formula for $\boldsymbol{b}_r$ involving the~$\0_k$'s and $\LL_0$.
 
 In Section \ref{preuvethm1}, we will see how this formula entails Theorem \ref{thm1}.
 
\subsection{A formula for $\boldsymbol{b}_r$}
\label{aformulaforbr}

By Theorem \ref{elementspropresdeL} and \eqref{LetL0}, we know that for every $\lambda \in \delta$ the unit circle in $\C$, $(\lambda - \LL_0)^{-1}$ exists.

Let $f(\lambda,t)$ be a formal power series on $t$ with values in $\End( L^2(\R^{2n}, \E_{x_0}))$:
\begin{equation}
f(\lambda,t) = \sum_{r=0}^{+\infty} t^r f_r(\lambda) \text{ with } f_r(\lambda) \in\End( L^2(\R^{2n}, \E_{x_0})).
\end{equation}

Consider the equation of formal power series on $t$ for $\lambda \in \delta$:
\begin{equation}
\left( \lambda - \LL_0 - \sum_{r=1}^{+\infty} t^r \0_r \right) f(\lambda,t) = \mathrm{Id}_{L^2(\R^{2n}, \E_{x_0})}.
\end{equation}
We then find that
\begin{equation}
\begin{aligned}
&f_0(\lambda) = (\lambda - \LL_0)^{-1}, \\
&f_r(\lambda) = (\lambda - \LL_0)^{-1} \sum_{j=1}^r \0_j f_{r-j}(\lambda).
\end{aligned}
\end{equation}
Thus by \eqref{annulation01} and by induction,
 \begin{equation}
\label{formulefr}
f_r(\lambda) = \left(\sum\limits_{\substack{r_1 +\cdots +r_k=r \\ r_j \geq 2}}  (\lambda - \LL_0)^{-1} \0_{r_1}\dots (\lambda - \LL_0)^{-1} \0_{r_k} \right)(\lambda - \LL_0)^{-1}.
\end{equation}

\begin{defn}
Following \cite[(4.1.91)]{ma-marinescu}, we define $\FF_r$ by
\begin{equation}
\label{defnFreq}
\FF_r = \frac{1}{2\pi \ic} \int_\delta f_r(\lambda) d\lambda,
\end{equation}
and we denote by $\FF_r(Z,Z')$ its smooth kernel with respect to $dv_{TX}(Z')$.
\end{defn}

\begin{thm}
\label{formulebrthm}
The following equation holds:
\begin{equation}
\label{formulebr}
\boldsymbol{b}_r(x_0) = \FF_{2r}(0,0).
\end{equation}
\end{thm}
\begin{proof}
This formula follows from \cite[Theorem 8.1.4]{ma-marinescu}, as \cite[(4.1.97)]{ma-marinescu} follows from \cite[Theorem 4.1.24]{ma-marinescu}, remembering that in our situation, the Bergman kernel $P_p$ is not supported  in degree~0.
 \end{proof}

\subsection{Proof of Theorem \ref{thm1}}
\label{preuvethm1}

Let $T_{\mathbf{r}} (\lambda)=  (\lambda - \LL_0)^{-1} \0_{r_1}\dots (\lambda - \LL_0)^{-1} \0_{r_k} (\lambda - \LL_0)^{-1}$ be the term in the sum~\eqref{formulefr} corresponding to $\mathbf{r}=(r_1,\dots,r_k)$. Let $N^\perp$ be the orthogonal of $N$ in $L^2(\R^{2n}, \E_{x_0})$, and $P^{N^\perp}$ be the associated orthogonal projector. In $T_{\mathbf{r}} (\lambda)$, each term $(\lambda - \LL_0)^{-1}$ can be decomposed as
\begin{equation}
\label{decompodelaresolvante}
(\lambda - \LL_0)^{-1} = (\lambda - \LL_0)^{-1} P^{N^\perp} + \frac{1}{\lambda} P^N.
\end{equation}
Set 
\begin{equation}
 L^{N^\perp}(\lambda)=(\lambda - \LL_0)^{-1} P^{N^\perp}, \qquad L^N(\lambda)=\frac{1}{\lambda} P^N.
 \end{equation}
By \eqref{LetL0}, $\LL_0$ preserves the degree, and thus so do $(\lambda - \LL_0)^{-1}$, $L^{N^\perp}$ and $L^N$.
 
 For $\eta=(\eta_1,\dots,\eta_{k+1})\in\{N,N^\perp \}^{k+1}$, let
\begin{equation}
T_{\mathbf{r}}^\eta(\lambda) = L^{\eta_1}(\lambda) \0_{r_1} \dots L^{\eta_k}(\lambda) \0_{r_k}L^{\eta_{k+1}}(\lambda).
\end{equation}
We can decompose:
\begin{equation}
T_{\mathbf{r}} (\lambda) = \sum_{\eta=(\eta_1,\dots,\eta_{k+1})} T_{\mathbf{r}}^\eta(\lambda),
\end{equation}
and by \eqref{formulefr} and \eqref{defnFreq}
\begin{equation}
\label{expressionFr}
\FF_{2r} = \frac{1}{2\pi \ic} \sum\limits_{\substack{ r_1 +\cdots +r_k=2r \\ (\eta_1,\dots,\eta_{k+1})}}  \int_\delta T_{\mathbf{r}}^\eta (\lambda) d\lambda.
\end{equation}

Note that $  L^{N^\perp}(\lambda)$ is an holomorphic function of $\lambda$, so
\begin{equation}
\int_\delta L^{N^\perp}(\lambda) \0_{r_1} \dots L^{N^\perp}(\lambda) \0_{r_k}L^{N^\perp}(\lambda) d\lambda =0.
\end{equation}
Thus, in \eqref{expressionFr}, every non-zero term that appears contains at least one $L^N (\lambda)$:
\begin{equation}
\label{ilyaduPn}
\int_\delta T_{\mathbf{r}}^\eta (\lambda) d\lambda \neq 0 \Rightarrow \text{there exists } i_0 \text{ such that }\eta_{i_0} = N.
\end{equation}

Now fix $k$ and $j$ in $\N$. Let $s\in L^2(\R^{2n}, \E_{x_0})$ be a form of degree $2j$, $\mathbf{r} \in (\N\setminus \{0,1\})^k$ such that $\sum_i r_i =2r$ and $\eta=(\eta_1,\dots,\eta_{k+1})\in\{N,N^\perp \}^{k+1}$ such that there is a $i_0$ satisfying $\eta_{i_0} = N$. We want to find a necessary condition for $I_{2j} T_{\mathbf{r}}^\eta (\lambda) I_{2j}s$ to be non-zero.

Suppose then that  $I_{2j} T_{\mathbf{r}}^\eta (\lambda) I_{2j}s\neq 0$. Since $L^{\eta_{i_0}} =  \frac{1}{\lambda} P^N$, and $N$ is concentrated in degree 0, we must have
$$ \deg \left( \0_{r_{i_0}} L^{\eta_{i_0+1}}(\lambda) \0_{r_{i_0+1}} \dots L^{\eta_k}(\lambda) \0_{r_k}L^{\eta_{k+1}}(\lambda)I_{2j}s \right) =0,$$
but each $L^{\eta_i}(\lambda)$ preserves the degree, and by Theorem \ref{dldeLt} each $\0_{r_i}$ lowers the degree at most by 2, so
$$0= \deg \left( \0_{r_{i_0}} L^{\eta_{i_0+1}}(\lambda) \0_{r_{i_0+1}} \dots L^{\eta_k}(\lambda) \0_{r_k}L^{\eta_{k+1}}(\lambda)I_{2j}s \right) \geq 2j - 2(k-i_0+1),$$
and thus
\begin{equation}
\label{majoration1}
2j \leq 2(k-i_0+1).
\end{equation}

Similarly, $L^{\eta_1}(\lambda) \0_{r_1} \dots L^{\eta_k}(\lambda) \0_{r_k}L^{\eta_{k+1}}(\lambda)I_{2j}s$ must have a non-zero component in degree $2j$ and by Theorem~\ref{dldeLt} each $\0_{r_i}$ rises the degree at most by 2, so $2j$ must be less or equal to the number of $\0_{r_i}$'s appearing before $\0_{r_{i_0}}$, that is
 \begin{equation}
 \label{majoration2}
 2j \leq 2(i_0 -1).
 \end{equation}
With \eqref{majoration1} and \eqref{majoration2}, we find
\begin{equation}
\label{majoration3}
 4j \leq 2k.
 \end{equation}
Finally, since for every $i$, $r_i \geq 2$ and $\sum_{i=1}^k r_i = 2r$, we have $2k \leq 2r$, and thus
\begin{equation}
\label{majoration4}
 4j \leq 2k \leq 2r.
 \end{equation}

Consequently, if $r<2j$ we have $I_{2j} T_{\mathbf{r}}^\eta (\lambda) I_{2j}=0$, and by \eqref{expressionFr}, we find $I_{2j} \FF_{2r} I_{2j}=0$. Using Theorem \ref{formulebrthm}, we find 
$$I_{2j} \boldsymbol{b}_r I_{2j}=0,$$
 which, combined with Theorem \ref{DLM2006}, entails the first part of Theorem \ref{thm1}.

For the second part of this theorem, let us assume that we are in the limiting case where $r=2j$. We also suppose that $j\geq1$, because in the case $j=0$, \cite[(8.1.5)]{ma-marinescu} implies that $\boldsymbol{b}_0(x_0)=\FF_0(0,0)=I_0\PP(0,0)= I_0$, so Theorem \ref{thm1} is true for $j=0$.

 In $I_{2j} \FF_{4j} I_{2j}$, there is only one term satisfying equations \eqref{majoration1}, \eqref{majoration2} and \eqref{majoration4}: first we see that  \eqref{majoration4} imply that $r=k=2j$ and for all $i$, $r_i=2$, while \eqref{majoration1} and \eqref{majoration2} imply that the $i_0 $ such that $\eta_{i_0} = N$ is unique and equal to $j$. Moreover, since the degree must decrease by $2j$ and then increase by $2j$ with only $k=2j$ $\0_{r_i}$'s available, only $\0_2^{+2}$ and $\0_2^{-2}$ appear in $I_{2j} \FF_{4j} I_{2j}$, and not $\0_2^0$. To summarize:

\begin{align}
\label{I2jF4jI2j}
I_{2j} \FF_{4j} I_{2j}&=\frac{1}{2\pi \ic} \int_\delta I_{2j} \left((\lambda - \LL_0)^{-1} P^{N^\perp} \0_2^{+2} \right)^j \frac{1}{\lambda} P^N \left(\0_2^{-2}  (\lambda - \LL_0)^{-1} P^{N^\perp}\right)^j I_{2j} d\lambda \notag \\
&=I_{2j} \left( \LL_0{}^{-1}P^{N^\perp} \0_2^{+2}\right)^j P^N \left(\0_2^{-2}\LL_0{}^{-1}P^{N^\perp}\right)^j I_{2j}\notag \\
&=I_{2j}  \left( \LL_0{}^{-1} \0_2^{+2}\right)^j P^N \left(\0_2^{-2}\LL_0{}^{-1} \right)^j  I_{2j},
\end{align}
because by \eqref{LetL0}, $L^2(\R^{2n},(\Lambda^{0,>0} (T^*X)\otimes E)_{x_0}) \subset N^\perp$, so we can remove the $P^{N^\perp}$'s.

Let $A=I_{2j}\left( \LL_0{}^{-1}\0_2^{+2}\right)^j P^N$. Since $(\0_2^{+2})^*=\0_2^{-2}$ (see Proposition \ref{expressionprecisedeO2}) and $\LL_0$ is self-adjoint, the adjoint of $A$ is $A^*=P^N\left( \0_2^{-2} \LL_0{}^{-1}\right)^j  I_{2j}$, and thus 
\begin{equation}
\label{F4jetA}
I_{2j} \FF_{4j} I_{2j} = A A^*.
\end{equation}

Recall that $P^N = \PP I_0$ (see \eqref{PNetP}). Let $s\in L^2(\R^{2n}, E_{x_0})$, since $\LL_0 = \LL +4\pi \num$ and $\LL\PP s=0$, the term $(\PP s)\RR_{x_0}$ is an eigenfunction of $\LL_0$ for the eigenvalue $2\times 4\pi$. Thus, we get
\begin{equation}
\LL_0{}^{-1} \0_2^{+2} P^Ns =\LL_0{}^{-1} \0_2^{+2} \PP s = \LL_0{}^{-1}\left( (\PP s)\RR_{x_0} \right) =\frac{1}{4\pi}\frac{1}{2} \RR_{x_0} \PP s.
\end{equation}
 Now, an easy induction shows that
\begin{equation}
\label{calculA}
A= \frac{1}{(4\pi)^j}\frac{1}{2\times 4 \times\cdots\times 2j} I_{2j} \RR_{x_0}^j \PP =  \frac{1}{(4\pi)^j}\frac{1}{2^j j!} I_{2j} \RR_{x_0}^j\PP.
\end{equation}

Let $A(Z,Z')$ and $A^*(Z,Z')$ be the smooth kernels of $A$ and $A^*$ with respect to $dv_{TX}(Z')$. 
By~\eqref{F4jetA}, $I_{2j} \FF_{4j} I_{2j}(0,0) = \int_{\R^{2n}}A(0,Z) A^*(Z,0)dZ$. Thanks to 
\begin{equation}
\label{P(0,0)}
\int_{\R^{2n}} \PP(0,Z) \PP(Z,0)dZ=(\PP \circ\PP)(0,0)=\PP(0,0)=1
\end{equation}
and \eqref{calculA}, we find \eqref{2epartiethm1}.


\section{The second coefficient in the asymptotic expansion}
\label{2ndcoeff}

In this section, we prove Theorem \ref{thm2}. Using \eqref{formulebr}, we know that 
\begin{equation}
I_{2j} \boldsymbol{b}_{2j+1} I_{2j}(0,0)=I_{2j} \FF_{4j+2} I_{2j}(0,0).
\end{equation}

 In Section \ref{decompodupb}, we decompose this term into three terms, and then in Sections \ref{onlyO2} and \ref{lesdeuxautres} we handle them separately. 
 
 For the rest of the section we fix $j \in \llbracket 0, n \rrbracket$. For every smoothing operator $F$ acting on $L^2(\R^{2n},\E_{x_0})$ that appears in this section, we will denote by $F(Z,Z')$ its smooth kernel with respect to $dv_{TX}(Z')$.

\subsection{Decomposition of the problem}
\label{decompodupb}

Applying inequality  \eqref{majoration4} with $r=2j+1$, we see that in $I_{2j} \FF_{4j+2} I_{2j}$, the non-zero terms $\int_\delta T^\eta_{\mathbf{r}}(\lambda) d\lambda$ appearing in decomposition \eqref{expressionFr} satisfy $k=2j$ or $k=2j+1$. Since $\sum_i r_i =4j+2$ and $r_i \geq 2$, we see that in $I_{2j} \FF_{4j+2} I_{2j}$ there are three types of terms $T^\eta_{\mathrm{r}}(\lambda)$ with non-zero integral, in which:
\begin{itemize}[label=$\bullet$]
\item for $k=2j$: 
\begin{itemize}
\item there are $2j-2$ $\0_{r_i}$'s equal to $\0_2$ and 2 equal to $\0_3$: we will denote by $\mathrm{I}$ the sum of these terms,
 \item there are $2j-1$ $\0_{r_i}$'s equal to $\0_2$ and 1 equal to $\0_4$: we will denote by $\mathrm{II}$ the sum of these terms,
 \end{itemize}
 \item for $k=2j+1$: 
 \begin{itemize}
 \item all the $\0_{r_i}$'s are equal to $\0_2$: we will denote by $\mathrm{III}$ the sum of these terms.      
 \end{itemize}
\end{itemize}
We thus have a decomposition 
\begin{equation}
\label{decompo2eterme}
I_{2j} \FF_{4j+2} I_{2j} =  \mathrm{I}+\mathrm{II}+\mathrm{III}.
\end{equation}

\begin{rem}
\label{remO3O4}
Note that for the two sums I and II to be non-zero, we must have $j \geq 1$. Moreover, in the two first cases, as $k=2j$, by the same reasoning as in Section \ref{preuvethm1}, \eqref{majoration1} and \eqref{majoration2} imply that the $i_0 $ such that $\eta_{i_0} = N$ is unique and equal to $j$, and that only $\0_2^{\pm2}$, $\0_3^{\pm2}$ and $\0_4^{\pm2}$ appear in I and II, and not some $\0_{r_i}^0$.
\end{rem}

\subsection{The term involving only $\0_2$}
\label{onlyO2}

\begin{lemme}
In any term $T^\eta_{\mathrm{r}}(\lambda)$ appearing in the term $\mathrm{III}$ (with non-vanishing integral), the  $i_0 $ such that $\eta_{i_0} = N$ is unique and equal to $j$ or $j+1$. If we denote by $\mathrm{III}_a$ and $\mathrm{III}_b$ the sum of the terms corresponding to these two cases, we have:
\begin{equation}
\label{IIIaetIIIb}
\begin{aligned}
&\mathrm{III}_a = \sum_{k=0}^j I_{2j} (\LL_0{}^{-1} \0_2^{+2})^{j-k}(\LL_0{}^{-1} \0_2^{0})(\LL_0{}^{-1} \0_2^{+2})^{k} P^N( \0_2^{-2}\LL_0{}^{-1})^j  I_{2j},\\
&\mathrm{III}_b =(\mathrm{III}_a)^* ,\\
&\mathrm{III}=\mathrm{III}_a+\mathrm{III}_b.
\end{aligned}
\end{equation}
\end{lemme}

\begin{rem}
\label{L0-1estdef}
For the same reason as for~\eqref{I2jF4jI2j}, we have removed the $P^{N^\perp}$'s in \eqref{IIIaetIIIb} without getting any problem with the existence of $\LL_0{}^{-1}$.
\end{rem}

\begin{proof}
Fix a term $T^\eta_{\mathrm{r}}(\lambda)$ appearing in the term $\mathrm{III}$ with non-vanishing integral. Using again the same reasoning as in Section \ref{preuvethm1}, we see that there exists at most two indices $i_0 $ such that $\eta_{i_0} = N$, and that they are in $\{j,j+1\}$. Indeed, with only $2j+1$ $\0_{r_i}$'s at our disposal,  we need $j$ of them before the first $P^N$, and $j$ after the last one.

Now, the only possible term with $\eta_j = \eta_{j+1}=N$ is:
$$(\LL_0{}^{-1} \0_2^{+2})^j P^N\0_2^0 P^N( \0_2^{-2}\LL_0{}^{-1})^j.$$
To prove that this term is vanishing, we will use \cite{MR2876259}. By \eqref{O20}, \cite[(3.13),(3.16b)]{MR2876259} and \cite[(4.1a)]{MR2876259} we see that $\PP\0_2^0 \PP =0$, and so
\begin{equation}
\label{PNO20PN=0}
P^N\0_2^0 P^N = \PP \0_2^0 \PP I_0 = 0,
\end{equation}
we have proved the first part of the lemma.

The second part follows from the reasoning made at the beginning of this proof, and the facts that $i_0$ is unique, $\0_2^0$ is self-adjoint and $(\LL_0{}^{-1} \0_2^{+2})^*=\0_2^{-2}\LL_0{}^{-1}$.
 \end{proof}

Let us compute the term that appears in \eqref{IIIaetIIIb}:
\begin{equation}
\label{defIIIak}
 \mathrm{III}_{a,k} := I_{2j} (\LL_0{}^{-1} \0_2^{+2})^{j-k}(\LL_0{}^{-1} \0_2^{0})(\LL_0{}^{-1} \0_2^{+2})^{k} P^N( \0_2^{-2}\LL_0{}^{-1})^j  I_{2j}.
 \end{equation}

With \eqref{calculA}, we know that
\begin{equation}
\label{commavant}
P^N( \0_2^{-2}\LL_0{}^{-1})^j  I_{2j} = \frac{1}{(4\pi)^j}\frac{1}{2^j j!} \PP \left(\RR_{x_0}^j\right) ^* I_{2j},
\end{equation}
and
\begin{align}
I_{2j} (\LL_0{}^{-1} \0_2^{+2})^{j-k}&(\LL_0{}^{-1} \0_2^{0})(\LL_0{}^{-1} \0_2^{+2})^{k} P^N \notag \\
&=\frac{1}{(4\pi)^k}\frac{1}{2^k k!} I_{2j} (\LL_0{}^{-1} \0_2^{+2})^{j-k}\LL_0{}^{-1}  (\0_2^{0}\RR_{x_0}^k\PP) I_0.
\label{avec020}
\end{align}
Let 
\begin{equation}
\begin{aligned}
& R_{k\bar{m}\ell \bar{q}} = \left \langle R^{TX}\left( \derpar{}{z_k} , \derpar{}{\bar{z}_m} \right) \derpar{}{z_\ell}, \derpar{}{\bar{z}_q} \right \rangle_{x_0}, \\
& R^E_{k\bar{\ell}} = R^E_{x_0}\left( \derpar{}{z_k} , \derpar{}{\bar{z}_\ell} \right).
\end{aligned}
\end{equation}

By \cite[Lemma 3.1]{MR2876259}, we know that 
\begin{equation}
\label{symetriesR}
R_{k\bar{m}\ell \bar{q}}=R_{\ell\bar{m}k \bar{q}}=R_{k\bar{q}\ell \bar{m}}=R_{\ell\bar{q}k \bar{m}}\text{ and }r^X_{x_0}=8R_{m\bar{m}q\bar{q}}.
\end{equation}

Once again, our $\0_2^0$ correspond to the $\0_2$ of \cite{MR2876259} (see \eqref{O20} and \cite[(3.13),(3.16b)]{MR2876259}), so we can use \cite[(4.6)]{MR2876259} to get:
\begin{equation}
\label{O20P}
\0_2^0\RR_{x_0}^k \PP = \left( \frac{1}{6}b_mb_qR_{k\bar{m}\ell \bar{q}}z_kz_\ell + \frac{4}{3}b_q R_{\ell \bar{k}k \bar{q}}z_\ell - \frac{\pi}{3}b_qR_{k\bar{m}\ell \bar{q}}z_k z_\ell \bar{z}'_m +b_q R^E_{\ell \bar{q}} z_\ell \right)\RR_{x_0}^k \PP,
\end{equation}

Set
\begin{equation}
\begin{aligned}
&a=\frac{1}{6}b_mb_qR_{k\bar{m}\ell \bar{q}}z_kz_\ell, & \qquad b=\frac{4}{3}b_q R_{\ell \bar{k}k \bar{q}}z_\ell, \\
&c= - \frac{\pi}{3}b_qR_{k\bar{m}\ell \bar{q}}z_k z_\ell \bar{z}'_m, &Ê d=b_q R^E_{\ell \bar{q}} z_\ell. \,\,\,\,\quad
\end{aligned}
\end{equation}
Thanks to \eqref{LetL0}, \eqref{vep} and \eqref{O20P}, we find
\begin{equation}
\label{L0-1RkO20P}
\LL_0{}^{-1}  \0_2^{0}\RR_{x_0}^{ k}\PP I_0 = \left( \frac{a}{4\pi(2+2k)}+\frac{b+c+d}{4\pi(1+2k)} \right)\RR_{x_0}^{ k}\PP I_0,
\end{equation}
and by induction, \eqref{avec020} becomes

\begin{align}
I_{2j} (\LL_0{}^{-1} \0_2^{+2})^{j-k}&(\LL_0{}^{-1} \0_2^{0})(\LL_0{}^{-1} \0_2^{+2})^{k} P^N \\
&= \frac{1}{(4\pi)^{j+1}}\frac{1}{2^k k!} I_{2j}\RR_{x_0}^{ j-k}\left( \frac{a}{(2+2k)\cdots(2+2j)}+\frac{b+c+d}{(1+2k)\cdots(1+2j)} \right)\RR_{x_0}^{ k}\PP I_0. \notag
\end{align}

\begin{lemme}
We have:
\begin{equation}
\begin{aligned}
&(a\RR_{x_0}^{ k}\PP)(0,Z) = \frac{1}{6} r^X_{x_0} \RR_{x_0}^{ k}\PP(0,Z),
&\quad (b\RR_{x_0}^{ k}\PP)(0,Z)= -\frac{1}{3} r^X_{x_0} \RR_{x_0}^{ k}\PP(0,Z), \\
&(c\RR_{x_0}^{ k}\PP)(0,Z)=0,
&\quad (d\RR_{x_0}^{ k}\PP)(0,Z) = -2 R^E_{q\bar{q}}\RR_{x_0}^{ k}\PP(0,Z).
\end{aligned}
\end{equation}
\end{lemme}
\begin{proof}
This lemma is a consequence of the relations \eqref{commutation} and \eqref{symetriesR}. For instance, we will compute $(b\RR_{x_0}^{ k}\PP)(0,Z)$, the other terms are similar.
\begin{align*}
(b\RR_{x_0}^{ k}\PP)(0,Z) &= \left( \frac{4}{3}b_q R_{\ell \bar{k}k \bar{q}}z_\ell \RR_{x_0}^{ k}\PP \right)(0,Z) \\
&=  \frac{4}{3} R_{\ell \bar{k}k \bar{q}}\RR_{x_0}^{ k} \left(\left( z_\ell b_q -2 \delta_{\ell q} \right) \PP\right)(0,Z) \\
&= -\frac{8}{3} R_{\ell \bar{k}k\bar{\ell}} \RR_{x_0}^{ k}\PP(0,Z)=-\frac{1}{3} r^X_{x_0} \RR_{x_0}^{ k}\PP(0,Z).
\end{align*}
 \end{proof}

Using \eqref{P(0,0)}, \eqref{defIIIak}, \eqref{commavant} and \eqref{L0-1RkO20P}, we find
 \begin{equation}
 \label{calculavec1O20}
 \mathrm{III}_{a,k}(0,0)\! = \! I_{2j}\CC_j(j)\RR_{x_0}^{ j-k} \! \left[ \frac{1}{6}\left(\CC_{j+1}(j+1)\!-\!\frac{\CC_j(k)}{2\pi(2k+1)}\!\right)r^X_{x_0}\! -\! \frac{\CC_j(k)}{2\pi(2k+1)}R^E_{q\bar{q}} \right] \!\RR_{x_0}^k \left(\RR_{x_0}^{ j}\right) ^*\! \! I_{2j}.
 \end{equation}
Notice that $2R^E_{q\bar{q}}=R^E_{x_0}\left( \sqrt{2}\derpar{}{z_q} , \sqrt{2}\derpar{}{\bar{z}_q} \right)=R^E_{x_0}(w_q,\bw_q)=\ic R^E_{\Lambda,x_0}$ by definition. Consequently, 
 \begin{equation}
 \label{III}
 \begin{aligned}
& \mathrm{III}_a(0,0) = \\
&I_{2j}\CC_j(j)  \! \sum_{k=0}^j \RR_{x_0}^{ j-k}\left[ \frac{1}{6} \! \left( \! \CC_{j+1}(j+1) \! - \! \frac{\CC_j(k)}{2\pi(2k+1)} \! \right) \!r^X_{x_0} \! - \! \frac{\CC_j(k)}{4\pi(2k+1)}\ic R^E_{\Lambda,x_0} \! \right] \!\RR_{x_0}^k \left(\RR_{x_0}^{ j}\right) ^*\!  I_{2j}.
 \end{aligned}
 \end{equation}

\subsection{The two other terms}
\label{lesdeuxautres}

In this subsection, we suppose that $j \geq 1$ (cf. Remark \ref{remO3O4}). Moreover, the existence of any $\LL_0{}^{-1}$ appearing in this section follows from the reasoning done in Remark \ref{L0-1estdef}, and this operator will be used without further precision.

Due to \eqref{dldeLtprecis}, we have
\begin{equation}
\label{O3O4}
\0_3^{+2}=\der{}{t}\left( \Phi^{+2}_{E_0}(tZ) \right)\restreint{t=0}=z_i \derpar{\RR_\cdot}{z_i}(0)+\bar{z}_i \derpar{\RR_\cdot}{\bar{z}_i}(0) \text{ and}
\end{equation}
\begin{equation}
\label{expression04+2}
\0_4^{+2}=\frac{z_iz_j}{2} \derpar{^2 \RR_\cdot}{z_i\partial z_j}(0)+z_i\bar{z}_j\derpar{^2 \RR_\cdot}{z_i\partial \bar{z}_j}(0)+\frac{\bar{z}_i \bar{z}_j }{2}\derpar{^2 \RR_\cdot}{\bar{z}_i\partial \bar{z}_j}(0).
\end{equation}

The sum I can be decomposed inyo 3 \textquoteleft sub-sums\textquoteright : I${}_a$, I${}_b$ and I${}_c$  in which the two $\0_3$'s appearing are respectively both at the left of $P^N$, either side of $P^N$ or both at the right of $P^N$ (see Remark~\ref{remO3O4}). As usual, we have $\mathrm{I}_c =( \mathrm{I}_a)^*$.

In the same way, we can decompose $\mathrm{II}=\mathrm{II}_a+\mathrm{II}_b$: in $\mathrm{II}_a$ the $\0_4$ appears at the left of $P^N$, and in $\mathrm{II}_b$ at the right of $P^N$. Once again, $\mathrm{II}_b =( \mathrm{II}_a)^*$.

\subsubsection*{Computation of $\mathrm{I}_b(0,0)$}

To compute I${}_b$, we first compute the value at $(0,Z)$ of the kernel of
\begin{equation}
\label{defnAk}
 A_k := I_{2j} (\LL_0{}^{-1} \0_2^{+2})^{j-k-1} (\LL_0{}^{-1} \0_3^{+2}) (\LL_0{}^{-1} \0_2^{+2})^k \PP I_0.
 \end{equation}

By \eqref{calculA} and \eqref{O3O4},
\begin{align}
\label{Ak}
A_k &= I_{2j} (\LL_0{}^{-1} \0_2^{+2})^{j-k-1} (\LL_0{}^{-1} \0_3^{+2}) \frac{1}{(4\pi)^k}\frac{1}{2^k k!} \RR_{x_0}^{ k} \PP I_0 \notag \\
&=\frac{1}{(4\pi)^k}\frac{1}{2^k k!} I_{2j} (\LL_0{}^{-1} \0_2^{+2})^{j-k-1}  \LL_0{}^{-1} \left[ z_i \derpar{\RR_{\cdot}}{z_i}(0)+\bar{z}_i \derpar{\RR_{\cdot}}{\bar{z}_i}(0)  \right]   \RR_{x_0}^{ k} \PP I_0.
\end{align}
Now by Theorem \ref{elementspropresdeL}, if $s \in N$, then $z_i s \in N$, so by the same calculation as in \eqref{calculA},

\begin{align}
\label{avecz}
& \frac{1}{(4\pi)^k}\frac{1}{2^k k!} \left(I_{2j}  (\LL_0{}^{-1} \0_2^{+2})^{j-k-1}  \LL_0{}^{-1} \left[ z_i \derpar{\RR_{\cdot}}{z_i}(0) \right]   \RR_{x_0}^{ k} \PP I_0 \right)(0,Z) \notag \\
&= \frac{1}{(4\pi)^j}\frac{1}{2^j j!} \left(I_{2j}  \left[ \RR_{x_0}^{ j-k-1}\derpar{\RR_{\cdot}}{z_i}(0)  \RR_{x_0}^{ k}  \right]z_i\PP I_0 \right)(0,Z) =0.
\end{align}

Now by \eqref{defbi} and the formula \eqref{formuleP}, we have 
\begin{equation}
\label{biP}
(b_i^+\PP)(Z,Z')=0 \text{ and } (b_i\PP)(Z,Z')=2\pi(\bar{z}_i - \bar{z}_i') \PP(Z,Z').
\end{equation}

Thus,
\begin{align}
\label{aveczbar}
&\frac{1}{(4\pi)^k}\frac{1}{2^k k!} \left( I_{2j} (\LL_0{}^{-1} \0_2^{+2})^{j-k-1}  \LL_0{}^{-1} \left[ \bar{z}_i \derpar{\RR_{\cdot}}{\bar{z}_i}(0)  \right]   \RR_{x_0}^{ k} \PP I_0\right)(Z,Z')  \\
&= \frac{1}{(4\pi)^k}\frac{1}{2^k k!} \left( I_{2j} (\LL_0{}^{-1} \0_2^{+2})^{j-k-1}  \LL_0{}^{-1} \left[ \derpar{\RR_{\cdot}}{\bar{z}_i}(0)    \RR_{x_0}^{ k}  \right] \left(\frac{b_i}{2\pi} + \bar{z}'_i\right)\PP I_0\right)(Z,Z') \notag \\
&= \frac{1}{(4\pi)^k}\frac{1}{2^k k!}  \left( I_{2j} (\LL_0{}^{-1} \0_2^{+2})^{j-k-1} \left[ \derpar{\RR_{\cdot}}{\bar{z}_i}(0)  \RR_{x_0}^{ k}  \right] \right.\notag \\
& \qquad \qquad \qquad \qquad \qquad \qquad \qquad \qquad \times \left.\left(\frac{1}{4\pi(2k+2+1)}\frac{b_i}{2\pi} +\frac{1}{4\pi(2k+2)} \bar{z}'_i\right)\PP I_0\right)(Z,Z') \notag\\
&= \frac{1}{(4\pi)^j}\frac{1}{2^j j!} \left( I_{2j}  \left[ \RR_{x_0}^{ j-k-1}\derpar{\RR_{\cdot}}{\bar{z}_i}(0)  \RR_{x_0}^{ k}  \right] \bar{z}'_i \PP I_0\right)(Z,Z')\notag \\
&  +  \frac{1}{(4\pi)^j}\frac{1}{2^k k!\prod_{k+1}^j(2\ell+1)} \left( I_{2j}  \left[ \RR_{x_0}^{ j-k-1}\derpar{\RR_{\cdot}}{\bar{z}_i}(0)  \RR_{x_0}^{ k}  \right] \frac{b_i}{2\pi}\PP I_0\right)(Z,Z').\notag
\end{align}
For the last two lines, we used that if $s \in N$, then $\LL(b_is)=4 \pi b_is$ (see Theorem \ref{elementspropresdeL}). Thus, by \eqref{defCjk} and \eqref{Ak}--\eqref{aveczbar}
\begin{align}
\label{Ak(0,Z)}
A_k(0,Z)&=\frac{1}{(4\pi)^k}\frac{1}{2^k k!} \left( I_{2j} (\LL_0{}^{-1} \0_2^{+2})^{j-k-1}  \LL_0{}^{-1} \left[ \bar{z}_i \derpar{\RR_{\cdot}}{\bar{z}_i}(0)  \right]   \RR_{x_0}^{ k} \PP I_0\right)(0,Z) \notag\\
&=  \Big(\CC_j(j)-\CC_j(k)\Big)  I_{2j}  \left[ \RR_{x_0}^{ j-k-1}\derpar{\RR_{\cdot}}{\bar{z}_i}(0)  \RR_{x_0}^{ k}  \right] \bar{z}_i\PP (0,Z)I_0.
\end{align}

We know that $(\bar{z}_i \PP)^* = z_i \PP$, and $\int_{\C^n} z_m \bar{z}_{q} e^{-\pi |z|^2}dZ = \frac{1}{\pi} \delta_{mq}$, so
\begin{align}
(A_{k_1} A_{k_2}{}^*)(0,0) &=  \frac{1}{\pi}  I_{2j}\left[ \Big(\CC_j(j)-\CC_j(k_1)\Big)\RR_{x_0}^{ j-k_1-1}\derpar{\RR_{\cdot}}{\bar{z}_i}(0)  \RR_{x_0}^{ k_1}  \right] \notag \\
& \qquad \qquad \qquad  \qquad \times\left[ \Big(\CC_j(j)-\CC_j(k_2)\Big) \RR_{x_0}^{ j-k_2-1}\derpar{\RR_{\cdot}}{\bar{z}_i}(0)  \RR_{x_0}^{ k_2}  \right]^* I_{2j}.
\end{align}
Finally,
\begin{multline}
\label{Ibprecis}
\mathrm{I}_b (0,0) =   \frac{1}{\pi} I_{2j} \left[ \sum_{k=0}^{j-1}  \Big(\CC_j(j)-\CC_j(k)\Big)\RR_{x_0}^{ j-k-1}\derpar{\RR_{\cdot}}{\bar{z}_i}(0)  \RR_{x_0}^{ k}  \right]  \\
 \times\left[\sum_{k=0}^{j-1}\Big(\CC_j(j)-\CC_j(k)\Big) \RR_{x_0}^{ j-k-1}\derpar{\RR_{\cdot}}{\bar{z}_i}(0)  \RR_{x_0}^{ k}  \right]^* I_{2j}.
\end{multline}

\subsubsection*{Computation of $\mathrm{I}_a(0,0)$ and $\mathrm{I}_c(0,0)$}

First recall that $\mathrm{I}_c(0,0)  = \left( \mathrm{I}_a(0,0) \right)^*$, so we just need to compute $\mathrm{I}_a(0,0)$. By the definition of $\mathrm{I}_a(0,0)$, for it to be non-zero, it is necessary to have $j\geq 2$, which will be assumed in this paragraph. Let
\begin{equation}
\label{defnAkl}
A_{k,\ell} := I_{2j} (\LL_0{}^{-1} \0_2^{+2})^{j-k-\ell-2} (\LL_0{}^{-1} \0_3^{+2}) (\LL_0{}^{-1} \0_2^{+2})^k(\LL_0{}^{-1} \0_3^{+2}) (\LL_0{}^{-1} \0_2^{+2})^\ell \PP I_0,
\end{equation}
the sum I${}_a(0,0)$ is then given by
\begin{equation}
\label{IaavecAkl}
\mathrm{I}_a(0,0)=\int_{\R^{2n}} \left(\sum_{k,\ell} A_{k,\ell}(0,Z)\right) \times \left( \frac{1}{(4\pi)^j} \frac{1}{2^j j!}I_{2j} \RR_{x_0}^{ j} \PP I_0 \right)^*(Z,0) dv_{TX}(Z).
\end{equation}

In the following, we will set 
\begin{equation}\tilde{b}_i:=\frac{b_i}{2\pi}.
\end{equation}
Using the same method as in \eqref{commutation}, \eqref{avecz}, \eqref{biP} and  \eqref{aveczbar}, we find that there exist constants  $C^1_{k,\ell}$, $C^2_{k,\ell}$ given by
\begin{equation}
\label{C12}
\begin{aligned}
&C^1_{k,\ell} =  \frac{1}{(4\pi)^{k+\ell+1}}\frac{1}{2^{k+\ell+1}(k+\ell +1)!}, \\
&C^2_{k,\ell} = \frac{1}{(4\pi)^{k+\ell+1}}\frac{1}{2^{\ell}\ell! \prod_{\ell+1}^{k+\ell+1}(2s+1)},
\end{aligned}
\end{equation}
such that
\begin{align}
\label{premiercalcul}
&(\LL_0{}^{-1} \0_3^{+2}) (\LL_0{}^{-1} \0_2^{+2})^k(\LL_0{}^{-1} \0_3^{+2}) (\LL_0{}^{-1} \0_2^{+2})^\ell \PP I_0 \\
&= \LL_0{}^{-1} \left\{ \derpar{\RR_{\cdot}}{z_i}(0)\RR_{x_0}^{ k}\derpar{\RR_{\cdot}}{z_{i'}}(0)\RR_{x_0}^{ \ell}C^1_{k,\ell}z_iz_{i'}+ \derpar{\RR_{\cdot}}{\bar{z}_i}(0)\RR_{x_0}^{ k}\derpar{\RR_{\cdot}}{z_{i'}}(0)\RR_{x_0}^{ \ell}C^1_{k,\ell}z_{i'} (\tilde{b}_i\!+\bar{z}'_i)\right.\notag \\
&\qquad \qquad  \qquad \qquad + \derpar{\RR_{\cdot}}{z_i}(0)\RR_{x_0}^{ k}\derpar{\RR_{\cdot}}{\bar{z}_{i'}}(0)\RR_{x_0}^{ \ell}z_i(C^2_{k,\ell}\tilde{b}_{i'}+C^1_{k,\ell}\bar{z}_{i'}') \notag  \\
&\qquad \qquad \qquad \qquad \qquad \qquad + \left. \derpar{\RR_{\cdot}}{\bar{z}_i}(0)\RR_{x_0}^{ k}\derpar{\RR_{\cdot}}{\bar{z}_{i'}}(0)\RR_{x_0}^{ \ell}(\tilde{b}_i+\bar{z}'_i)(C^2_{k,\ell}\tilde{b}_{i'}+C^1_{k,\ell}\bar{z}_{i'}') \right\}\PP I_0 \notag \\
&= \LL_0{}^{-1} \! \left\{ \derpar{\RR_{\cdot}}{z_i}(0)\RR_{x_0}^{ k}\derpar{\RR_{\cdot}}{z_{i'}}(0)\RR_{x_0}^{ \ell}C^1_{k,\ell}z_iz_{i'}  + \derpar{\RR_{\cdot}}{\bar{z}_i}(0)\RR_{x_0}^{ k}\derpar{\RR_{\cdot}}{z_{i'}}(0)\RR_{x_0}^{ \ell}C^1_{k,\ell}\Big(\tilde{b}_iz_{i'} \! + \! \frac{\delta_{ii'}}{\pi} \! + \! z_{i'}\bar{z}'_i \Big)\right. \notag \\
&\qquad \qquad \qquad+\derpar{\RR_{\cdot}}{z_i}(0)\RR_{x_0}^{ k}\derpar{\RR_{\cdot}}{\bar{z}_{i'}}(0)\RR_{x_0}^{ \ell}\left(C^2_{k,\ell} \Big(\tilde{b}_{i'}z_i +\frac{\delta_{ii'}}{\pi}\Big)+C^1_{k,\ell}z_i\bar{z}_{i'}'\right) \notag \\
&\qquad \qquad \qquad \qquad + \left. \derpar{\RR_{\cdot}}{\bar{z}_i}(0)\RR_{x_0}^{ k}\derpar{\RR_{\cdot}}{\bar{z}_{i'}}(0)\RR_{x_0}^{ \ell}\Big(C^2_{k,\ell}(\tilde{b}_i \tilde{b}_{i'}+ \bar{z}'_i \tilde{b}_{i'})+ C^1_{k,\ell}( \tilde{b}_i \bar{z}'_{i'}+\bar{z}'_i\bar{z}_{i'}')\Big) \right\}\PP I_0,\notag
\end{align}

Using Theorem \ref{elementspropresdeL}, \eqref{commutation} and \eqref{biP}, we see that there exist constants $C_{j,k,\ell}^i$, $i=3,\dots, 10$, such that 
\begin{equation}
\label{relationsC}
\begin{aligned}
& C^3_{j,k,\ell}=C^1_{k,\ell} \frac{1}{(4\pi)^{j-(k+\ell+1)}}\frac{1}{\prod_{k+\ell+2}^{j}(2s)},
\quad C^4_{j,k,\ell} =C^1_{k,\ell}  \frac{1}{(4\pi)^{j-(k+\ell+1)}}\frac{1}{\prod_{k+\ell+2}^{j}(2s+1)}, \\
& C^5_{j,k,\ell} =C^2_{k,\ell}  \frac{1}{(4\pi)^{j-(k+\ell+1)}}\frac{1}{\prod_{k+\ell+2}^{j}(2s+1)}, 
\quad C^6_{j,k,\ell} = C^2_{k,\ell} \frac{1}{(4\pi)^{j-(k+\ell+1)}}\frac{1}{\prod_{k+\ell+2}^{j}(2s)},
\end{aligned}
\end{equation}
and
\begin{align}
\label{deuxiemalcul}
&A_{k,\ell}(0,Z) =I_{2j}\left(\RR_{x_0}^{ {j-k-\ell-2}} \left \{ \derpar{\RR_{\cdot}}{\bar{z}_i}(0)\RR_{x_0}^{ k}\derpar{\RR_{\cdot}}{z_{i'}}(0)\Big(C^3_{j,k,\ell} \frac{\delta_{ii'}}{\pi}+C^4_{j,k,\ell}\tilde{b}_iz_{i'}\Big)  \right. \right. \\
& \quad+\derpar{\RR_{\cdot}}{z_i}(0)\RR_{x_0}^{ k}\derpar{\RR_{\cdot}}{\bar{z}_{i'}}(0)\Big(C^5_{j,k,\ell} \tilde{b}_{i'}z_i + \frac{\delta_{ii'}}{\pi}C^6_{j,k,\ell}\Big) \notag \\
& \quad  + \left. \left. \derpar{\RR_{\cdot}}{\bar{z}_i}(0)\RR_{x_0}^{ k}\derpar{\RR_{\cdot}}{\bar{z}_{i'}}(0)\Big(C^7_{j,k,\ell} \tilde{b}_i \tilde{b}_{i'}+ C^8_{j,k,\ell}\bar{z}_i \tilde{b}_{i'}+ C^9_{j,k,\ell} \tilde{b}_i \bar{z}_{i'}+C^{10}_{j,k,\ell}\bar{z}_i\bar{z}_{i'}\Big) \right\} \RR_{x_0}^{ \ell}\PP I_0\right)(0,Z) \notag \\
&=I_{2j} \RR_{x_0}^{ {j-k-\ell-2}} \left \{ \derpar{\RR_{\cdot}}{\bar{z}_i}(0)\RR_{x_0}^{ k}\derpar{\RR_{\cdot}}{z_{i}}(0)\frac{C^3_{j,k,\ell}-C^4_{j,k,\ell}}{\pi} \right. +\derpar{\RR_{\cdot}}{z_i}(0)\RR_{x_0}^{ k}\derpar{\RR_{\cdot}}{\bar{z}_{i}}(0)\frac{C^6_{j,k,\ell}-C^5_{j,k,\ell}}{\pi} \notag \\
& \quad \left. +\derpar{\RR_{\cdot}}{\bar{z}_i}(0)\RR_{x_0}^{ k}\derpar{\RR_{\cdot}}{\bar{z}_{i'}}(0)\Big(4 C^7_{j,k,\ell}-2 C^8_{j,k,\ell}-2 C^9_{j,k,\ell} +C^{10}_{j,k,\ell}\Big)\bar{z}_i\bar{z}_{i'}\right\} \RR_{x_0}^{ \ell}  \PP(0,Z) I_0.\notag 
\end{align}

Now with $\int \bar{z}_i\bar{z}_{i'} \PP(0,Z)\PP(Z,0) dZ = 0$, we can rewrite \eqref{IaavecAkl}:
\begin{multline}
\label{Iaaveclescstes}
\mathrm{I}_a(0,0)=  \frac{\CC_j(j)}{\pi}  I_{2j}  \sum_{k,\ell} \RR_{x_0}^{ {j-k-\ell-2}} \left\{ (C^3_{j,k,\ell}-C^4_{j,k,\ell})\derpar{\RR_{\cdot}}{\bar{z}_i}(0)\RR_{x_0}^{ k}\derpar{\RR_{\cdot}}{z_{i}}(0) \right.\\
\left. +(C^6_{j,k,\ell}-C^5_{j,k,\ell}) \derpar{\RR_{\cdot}}{z_i}(0)\RR_{x_0}^{ k}\derpar{\RR_{\cdot}}{\bar{z}_{i}}(0)\right\}\RR_{x_0}^{ \ell} \left( \RR_{x_0}^{ j}  \right)^* I_{2j}. 
\end{multline}

By \eqref{defCjk}, \eqref{C12} and \eqref{relationsC},
\begin{equation}
\label{C456}
\begin{aligned}
C^3_{j,k,\ell}=\CC_j(j),
&\qquad C^4_{j,k,\ell} =\CC_j (k+\ell+1), \\
C^5_{j,k,\ell} =\CC_j(\ell), 
&\qquad C^6_{j,k,\ell} =\CC_j(\ell)\prod_{s=k+\ell+2}^j \left(1+ \frac{1}{2s} \right).
\end{aligned}
\end{equation}
We can now write \eqref{Iaaveclescstes} more precisely:
\begin{align}
\label{Iaprecis}
\mathrm{I}_a(0,0)&=    \frac{\CC_j(j)}{\pi}   I_{2j}  \sum_{q=0}^{j-2}\sum_{m=0}^{q} \left \{ \Big(\CC_j(j)-\CC_j(q+1) \Big)  \RR_{x_0}^{j-(q+2)}\derpar{\RR_{\cdot}}{\bar{z}_i}(0)\RR_{x_0}^{q-m}\derpar{\RR_{\cdot}}{z_i}(0)\RR_{x_0}^{m} \right. \notag \\
& +\CC_j\left(m \right) \left[\prod_{q+2}^j\Big(1+\frac{1}{2s}\Big)-1 \right] \left. \RR_{x_0}^{j-(q+2)} \derpar{\RR_{\cdot}}{z_i}(0)\RR_{x_0}^{q-m}\derpar{\RR_{\cdot}}{\bar{z}_i}(0)\RR_{x_0}^{m} \right\}\left( \RR_{x_0}^{ j}  \right)^* I_{2j}. 
\end{align}

\subsubsection*{Computation of $\mathrm{II}(0,0)$}

Recall that $\mathrm{II}(0,0)=\mathrm{II}_a(0,0)+ \left( \mathrm{II}_a(0,0) \right)^*$. The computation of $\mathrm{II}_a(0,0)$ is very similar to the computation of $\mathrm{I}_a(0,0)$, and is simpler, so we will follow the same method.

Let
$$B_{k} := I_{2j} (\LL_0{}^{-1} \0_2^{+2})^{j-k-1} (\LL_0{}^{-1} \0_4^{+2}) (\LL_0{}^{-1} \0_2^{+2})^k \PP I_0,$$
the sum II${}_a(0,0)$ is then given by
\begin{equation}
\label{IIaavecAk}
\mathrm{II}_a(0,0)=\int_{\R^{2n}} \left(\sum_{k} B_{k}(0,Z)\right) \times \left( \frac{1}{(4\pi)^j} \frac{1}{2^j j!} I_{2j} \RR_{x_0}^{ j} \PP I_0 \right)^*(Z,0) dv_{TX}(Z).
\end{equation}

Using \eqref{expression04+2}, we can repeat what we have done for  \eqref{premiercalcul} and \eqref{deuxiemalcul}. We find that there is a constant $C$ (which we do not need to compute) such that 
\begin{multline}
B_k(0,Z) =  I_{2j}\left \{ \RR_{x_0}^{j-(k+1)} \derpar{^2\RR_\cdot}{ z_i \partial \bar{z}_{i}} (0) \RR_{x_0}^k \frac{\CC_j(j)-\CC_j(k)}{\pi}  \right.  \\
\left. +  \RR_{x_0}^{j-(k+1)} \derpar{^2\RR_\cdot}{\bar{z}_i \partial \bar{z}_{i'}} (0) \RR_{x_0}^k C\frac{\bar{z}_i \bar{z}_{i'}}{2} \right \} \PP(0,Z)I_0.
\end{multline}
Thus, we get 
\begin{equation}
\label{IIaprecis}
\mathrm{II}_a(0,0)= \frac{\CC_j(j)}{\pi}  I_{2j}  \sum_{k=0}^{j-1} \Big( \CC_j(j)-\CC_j(k) \Big) \RR_{x_0}^{j-(k+1)} \derpar{^2\RR_\cdot}{z_i \partial \bar{z}_{i}}(0) \RR_{x_0}^k \left( \RR_{x_0}^{ j}  \right)^* I_{2j} .
\end{equation}

\subsubsection*{Conclusion}

In order to conclude the proof of Theorem \ref{thm2}, we just have to put the pieces together. But before that, as we want to write the formulas in a more intrinsic way, we have to note that since we trivialized $\Wedge(T^*X)\otimes E$ with $\nabla^{\Wedge \otimes E}$, since $\bw_i = \sqrt{2}\derpar{}{\bar{z}_i}$, and thanks to \cite[(5.44),(5.45)]{MR2876259}, we have
\begin{align}
&\derpar{\RR_{\cdot}}{\bar{z}_i}(0) = \frac{1}{\sqrt{2}}\left( \nabla^{\Wedge \otimes E}_{\bw_i}\RR_{\cdot} \right)(x_0), \quad \derpar{\RR_{\cdot}}{z_i}(0) = \frac{1}{\sqrt{2}}\left( \nabla^{\Wedge \otimes E}_{w_i}\RR_{\cdot} \right)(x_0)\text{ and } \\
&\derpar{^2\RR_{\cdot}}{z_i \partial \bar{z}_{i}}(0)= -\frac{1}{4}(\Delta^{\Wedge \otimes E} \RR_{\cdot}) (x_0) .
\end{align}
With these remarks and  equations  \eqref{IIIaetIIIb},  \eqref{III}, \eqref{Ibprecis}, \eqref{Iaprecis}, \eqref{IIaprecis} used in decomposition \eqref{decompo2eterme}, we get Theorem \ref{thm2}.


\section{The third coefficient in the asymptotic expansion when the first two vanish}
\label{3rdcoeff}

In this section, we prove Theorem \ref{thm3}. Using \eqref{formulebr}, we know that 
\begin{equation}
I_{2j} \boldsymbol{b}_{2j+2} I_{2j}(0,0)=I_{2j} \FF_{4j+4} I_{2j}(0,0).
\end{equation}

 Here again, we will first decompose this term into several terms in Section \ref{decompodupb2}, and then in Sections \ref{typeI}, \ref{typeII} and \ref{typeIII} we handle them separately. 
 
 For the rest of the section we fix $j \in \llbracket 1, n \rrbracket$, and we suppose that 
 \begin{equation}
I_{2j} \boldsymbol{b}_{2j} I_{2j}(0,0)=I_{2j} \boldsymbol{b}_{2j+1} I_{2j}(0,0)=0.
\end{equation}
By Theorems \ref{thm1} and \ref{thm2}, this is equivalent to
\begin{equation}
\label{conditions}
\left\{\begin{aligned}
&\RR_x ^j=0  \\
&\T_0(j)=0.
\end{aligned}\right.
\end{equation}

  For every smoothing operator $F$ acting on $L^2(\R^{2n},\E_{x_0})$ that appears in this section, we will denote by $F(Z,Z')$ its smooth kernel with respect to $dv_{TX}(Z')$. Moreover, recall that every operators $A$ we have
   \begin{equation}
  \mathrm{Pos}[A]=AA^* \quad \text{and} \qquad \mathrm{Sym}[A]=A+A^*.
  \end{equation}
 
\subsection{Decomposition of the computation}
\label{decompodupb2}

With the same reasoning as in Section \ref{decompodupb}, we see that in the decomposition \eqref{expressionFr} of $I_{2j} \FF_{4j+4} I_{2j}$, the non-zero terms $\int_\delta T^\eta_{\mathbf{r}}(\lambda) d\lambda$ appearing  satisfy $k=2j,\, 2j+1$ or $2j+2$. Moreover, we can find the possible terms by adding one term to or modifying the subscript of the terms we mentioned in section \ref{decompodupb}. The list of possible terms is as follows.
\begin{enumerate}[label=\Roman*.]
\item The terms such that $k=2j+2$. 

Here, there are up to three indices $i$ such that $\eta_{i}=N$ is  and are in $\{j,j+1,j+2\}$. Moreover, the only $\0_\ell$'s appearing are some $\0_2$'s. The possibilities are now
\begin{enumerate}[label=I-\alph*)]
\item $2j+2$ times $\0_2^{\pm2}$,
\item $2j$ times $\0_2^{\pm2}$ and 2 times $\0_2^{0}$.
\end{enumerate}

\item The terms such that $k=2j+1$. 

Here, there are one or two indices $i$ such that $\eta_{i}=N$ is  and are in $\{j,j+1\}$, and there is exactly one $\0_\ell^0$ that appears in these terms. We regroup them in relation to the $\0_{r_i}$ that they contain.
\begin{enumerate}[label=II-\alph*)]
\item $2j$ times $\0_2^{\pm2}$ and 1 time $\0_4^{0}$,
\item $2j-1$ times $\0_2^{\pm2}$, 1 time $\0_2^0$ and 1 time $\0_4^{\pm2}$,
\item $2j-1$ times $\0_2^{\pm2}$, 1 time $\0_3^{\pm2}$ and 1 time $\0_3^0$,
\item $2j-2$ times $\0_2^{\pm2}$, 1 time $\0_2^0$ and 2 times $\0_3^{\pm2}$.
\end{enumerate}

\item The terms such that $k=2j$. 

Here, the $i_0$ such that $\eta_{i_0}=N$ is unique and equal to $j$, and no $\0_\ell^0$ appears in these terms. We regroup them in relation to the $\0_{r_i}$ that they contain.
\begin{enumerate}[label=III-\alph*)]
\item $2j-4$ times $\0_2^{\pm2}$ and 4 times $\0_3^{\pm2}$,
\item $2j-3$ times $\0_2^{\pm2}$, 2 times $\0_3^{\pm2}$ and 1 time $\0_4^{\pm2}$,
\item $2j-2$ times $\0_2^{\pm2}$ and 2 times $\0_4^{\pm2}$,
\item $2j-2$ times $\0_2^{\pm2}$, 1 time $\0_3^{\pm2}$ and 1 time $\0_5^{\pm2}$
\item $2j-1$ times $\0_2^{\pm2}$ and 1 time $\0_6^{\pm2}$.
\end{enumerate}
\end{enumerate}

This list seem quite long, but fortunately most of the terms will ultimately vanish due to the fact that they make appear some terms involved in $I_{2j} \boldsymbol{b}_{2j} I_{2j}$ and $I_{2j} \boldsymbol{b}_{2j+1} I_{2j}$.

In the sequel, the contribution to the third coefficient of the terms of type I-a), I-b), etc. will be denoted by $T_{\text{I-a)}}$, $T_{\text{I-b)}}$, etc.

\subsection{The terms of type I}
\label{typeI}

We first begin with the following observation.
\begin{lemme}
\label{A_a_i}
For any $j$-tuple $(a_1,\dots,a_j)$ of positive integers, we have
\begin{equation}
X_{(a_1,\dots,a_j)}:=I_{2j}\left(\prod_{i=1}^j\LL_0{}^{-a_i}\0_2^{+2}\right)P^N =0.
\end{equation}
\end{lemme}
\begin{proof}
This is an easy extension of the computation \eqref{calculA}, using the fact that $\RR^j_x=0$.
\end{proof}

\subsubsection*{The terms of type I-a)}

In these terms, only some $\0_2^{\pm2}$ appears, so there is either a unique $i_0$ such that $\eta_{i_0}=N$ which is then equal to $j$ or $j+2$, either exactly two such $i_0$'s which are then $j$ and $j+2$.

Each term that is in the second case is a sum of term of the form 
\begin{equation}
- X_{(a_1,\dots,a_j)}\0_2^{-2}\LL_0{}^{-b}\0_2^{+2}X_{(a'_1,\dots,a'_j)}^*
\end{equation}
with $a_i,a'_k,b \in \{1,2\}$ (exactly one is equal to 2). By Lemma \ref{A_a_i}, these terms vanish.

Now, each term in the first case is equal or adjoint to a term of the form
\begin{equation}
 I_{2j}\left(\prod_{i=1}^{j+2}\LL_0{}^{-1}\0_2^{\e_i}\right)\PP I_0 \left(I_{2j}\left(\LL_0{}^{-1}\0_2^{+2}\right)^j\PP I_0\right)^*,
\end{equation}
where  $\e_i\in\{-2,+2\}$ (exactly one of the $\e_i$'s  is equal to -2).  By Lemma, \ref{A_a_i}, these terms vanish.

Finally, every term of type I-a) vanishes and $T_{\text{I-a)}}=0$.

 \subsubsection*{The terms of type I-b)}
 
Using Lemma \ref{A_a_i} as above, we see that the only non-zero terms of this type satisfy that before the first index $i$ such that $\eta_i=N$ and after the last, there must be a $\0_2^0$ appearing. As a consequence, the cases where two or three $\eta_i$'s are equal to $N$ lead to vanishing terms. We now deal with the terms where $\eta_{j+1}=N$ and for $i\neq j+1$, $\eta_i = N^\perp$. Such terms are of the form
\begin{equation}
\Big(I_{2j} (\LL_0{}^{-1}\0_2^{+2})^{j-k}\LL_0{}^{-1}\0_2^{0}(\LL_0{}^{-1}\0_2^{+2})^{k} \PP \Big)\Big(I_{2j} (\LL_0{}^{-1}\0_2^{+2})^{j-k'}\LL_0{}^{-1}\0_2^{0}(\LL_0{}^{-1}\0_2^{+2})^{k'} \PP \Big)^*,
\end{equation}
 for $0\leq k,k' \leq j$. By the computations of Section \ref{onlyO2}, and in particular \eqref{calculavec1O20}, we find
 \begin{equation}
 \begin{aligned}
& I_{2j}\Big( (\LL_0{}^{-1}\0_2^{+2})^{j-k}\LL_0{}^{-1}\0_2^{0}(\LL_0{}^{-1}\0_2^{+2})^{k} \Big)\PP I_0= \\
& \qquad I_{2j} \RR_{x}^{j-k} \left[ \frac{1}{6} \! \left( \! \CC_{j+1}(j+1) \! - \! \frac{\CC_j(k)}{2\pi(2k+1)} \! \right) \!r^X_{x} \! - \! \frac{\CC_j(k)}{4\pi(2k+1)}\ic R^E_{\Lambda,x} \! \right]\RR_{x}^{k} \PP I_0.
 \end{aligned}
 \end{equation}
 
 Observe that $r^X$ commutes with $\RR$, and that $\RR^j=0$, so the contribution of the terms of type I-b) is finally $T_{\text{I-b)}}=\mathrm{Pos} \left[ \T'_3(j) \right]$.

\subsection{The terms of type II}
\label{typeII}

  \subsubsection*{The terms of type  II-a)}
 
 In these terms, there are either only $\0_2^{-2}$ appearing at the right of the first $P^N$ or only $\0_2^{+2}$ appearing at the left of the last $P^N$. Either way, all these terms vanish by Lemma \ref{A_a_i}. Hence $T_{\text{II-a)}}=0$.

   \subsubsection*{The terms of type  II-b)}
 
 For these terms, there are two possibilities. 
 
 Firstly, there are two indices $i$ such that $\eta_{i}=N$, and then they are equal to $j$ and $j+1$. In this case, either before the first $P^N$, either after the last, there are $j$ $\0_2^{+2}$'s (or $\0_2^{-2}$'s) that appear, so all these terms vanish. 
 
 Secondly, there is a unique $i_0$ such that $\eta_{i_0}=N$ and it is equal to $j$ or $j+1$. We denote by $S_1$ (resp. $S_2$) the sum of the term for which $i_0=j$ (resp. $i_0=j+1$). Then $S_1=S_2^*$ and
 \begin{equation}
 \begin{aligned}
 \label{S2}
 S_2&= \sum_{k,\ell} \left\{  I_{2j} (\LL_0{}^{-1} \0_2^{+2})^{j-k-1} (\LL_0{}^{-1} \0_4^{+2}) (\LL_0{}^{-1} \0_2^{+2})^k \PP I_0 \right \}\\
 & \qquad \qquad \qquad \qquad \qquad \qquad \qquad \times \left\{ I_{2j} (\LL_0{}^{-1}\0_2^{+2})^{j-\ell}\LL_0{}^{-1}\0_2^{0}(\LL_0{}^{-1}\0_2^{+2})^{\ell} \PP I_0 \right\}^* \\
 &= \left\{ \sum_{k} I_{2j} (\LL_0{}^{-1} \0_2^{+2})^{j-k-1} (\LL_0{}^{-1} \0_4^{+2}) (\LL_0{}^{-1} \0_2^{+2})^k \PP I_0 \right \}\\
 & \qquad \qquad \qquad  \qquad \qquad \qquad \times \left\{ \sum_{\ell} I_{2j} (\LL_0{}^{-1}\0_2^{+2})^{j-\ell}\LL_0{}^{-1}\0_2^{0}(\LL_0{}^{-1}\0_2^{+2})^{\ell} \PP I_0 \right\}^*.
 \end{aligned}
 \end{equation}
By \eqref{calculavec1O20} and \eqref{IIaprecis} we find that the contribution of the terms of type II-b), i.e. $S_1(0,0)+S_2(0,0)$, is $T_{\text{II-b)}}=\mathrm{Sym}\left[\T_2(j)\T'_3(j)^*\right]$.

\subsubsection*{The terms of type  II-c)} 
 
 The computation is the same as for terms of type II-b), except that in the case of a unique $i_0$ such that $\eta_{i_0}=N$, in \eqref{S2} we must replace $\0_4^{+2}$ by $\0_3^{+2}$ and $\0_2^0$ by $\0_3^0$. Recall that $A_k$ has been defined in \eqref{defnAk}.  By \eqref{Ak(0,Z)} and \eqref{conditions}, we find that the contribution of the terms of type II-c) is the symmetric operator associated to
\begin{equation}
\left\{ \sum_{k} A_k \right \}\left\{ \sum_{\ell} I_{2j} (\LL_0{}^{-1}\0_2^{+2})^{j-\ell}\LL_0{}^{-1}\0_3^{0}(\LL_0{}^{-1}\0_2^{+2})^{\ell} \PP I_0 \right\}^*.
\end{equation}
By \eqref{conditions} we get $T_{\text{II-c)}}=0$.

\subsubsection*{The terms of type  II-d)} 

Here again, we have the same possibilities concerning the indices $i$ such that $\eta_i=N$ as for terms of types II-b) or II-c). If there are two such indices, then they are equal to $j$ and $j+1$ and between the two corresponding $P^N$'s we will have the term $\0_2^0$. By \eqref{PNO20PN=0}, these terms vanish.

We now suppose that there is a unique $i_0$ such that $\eta_{i_0}=N$. Then $i_0=j$ or $j+1$. As $\RR^j_x=0$, any term where the two $\0_3$'s and the $\0_2^0$ appear on the same side of $P^N$ will vanish. A term where there is 1 $\0_3$ at the left and 1 $\0_3$ at the right of $P^N$ is equal or adjoint to
\begin{equation}
I_{2j}\left(\prod_{i=1}^{j+1}\LL_0{}^{-1}\0_{a_i}^{\e_i}\right)\PP I_0 \times A_k^*,
\end{equation}
where $a_i=2$ or $3$ and $\e_i=+2$ except for exactly one $i_1$ satisfying $a_{i_1}=2$ (for which $\e_{i_1}=0$). By \eqref{Ak(0,Z)} and \eqref{conditions}, the sum of this terms vanishes.

Finally, the only possibility is that the two $\0_3$'s  appear on the same side of $P^N$, and $\0_2^0$ on the other side. Recall that $A_{k,\ell}$ has been define in \eqref{defnAkl}. The sum of the remaining terms is equal to
\begin{equation}
\mathrm{Sym}\left[\left\{ \sum_{k,\ell} A_{k,\ell} \right \}\left\{ \sum_{m} I_{2j} (\LL_0{}^{-1}\0_2^{+2})^{j-m}\LL_0{}^{-1}\0_2^{0}(\LL_0{}^{-1}\0_2^{+2})^{m} \PP I_0 \right\}^*\right].
\end{equation}
As a consequence, the contribution of terms of type II-d) is $T_{\text{II-d)}}=\mathrm{Sym}\left[\T_1(j)\T'_3(j)^*\right]$.

\subsection{The terms of type III}
\label{typeIII}

The computations rely on similar arguments that the ones undertaken in Sections \ref{typeI} and \ref{typeII}. We will therefore give the contribution of each sub-type directly. 

\subsubsection*{The terms of type  III-a)} 

The contribution of these terms is $T_{\text{III-a)}}=\mathrm{Pos}\left[\T_1(j)\right]$.

\subsubsection*{The terms of type  III-b)} 

The contribution of these terms is $T_{\text{III-b)}}=\mathrm{Sym}\left[\T_1(j)\T_2(j)^*\right]$

\subsubsection*{The terms of type  III-c)} 

The contribution of these terms is $T_{\text{III-c)}}=\mathrm{Pos}\left[\T_2(j)\right]$.

\subsubsection*{The terms of type  III-d)} 
   
The sum of all these terms vanishes: $T_{\text{III-d)}}=0$.

\subsubsection*{The terms of type  III-e)} 

This terms vanishes, so that $T_{\text{III-e)}}=0$.

By all the computations of  Sections \ref{typeI}, \ref{typeII} and \ref{typeIII}, we get Theorem \ref{thm3}.


\end{document}